\documentclass[a4paper,twoside,11pt]
{amsart}
\usepackage{amssymb}
\usepackage{graphicx}

\newtheorem{thm}{Theorem}

\newtheorem{lem}{Lemma}
\newtheorem{prop}{Proposition}
\newtheorem{defn}{Definition}
\newtheorem{rem}{Remark}
\newtheorem{exm}{Example}

\newcommand{\g}{\mathfrak g}
\renewcommand{\ng}{\mathfrak n}
\newcommand{\ug}{\mathfrak u}

\newcommand{\sll}{\mathfrak{sl}}

\newcommand{\Hom}{\operatorname{Hom}}

\newcommand{\mg}{\mathfrak g}
\newcommand{\vf}{\varphi}
\newcommand{\gl}{\mathfrak{gl}}

\newcommand{\gr}{\operatorname{gr}}
\begin{document}
\title{On geometry of curves of flags of constant type}

\author
{Boris Doubrov
\address{Belarussian State University, Nezavisimosti Ave.~4, Minsk 220030, Belarus;
 E-mail: doubrov@islc.org}
\and Igor Zelenko
\address{Department of Mathematics, Texas A$\&$M University,
   College Station, TX 77843-3368, USA; E-mail: zelenko@math.tamu.edu}}

\begin{abstract}
We develop an algebraic version of Cartan method of equivalence or an analog of Tanaka prolongation for the (extrinsic) geometry of curves of flags of a vector space $W$ with respect to the action of a subgroup $G$ of the  $GL(W)$. Under some natural assumptions on the subgroup $G$ and on the flags, one can  pass from the filtered objects to the corresponding graded objects and describe the construction of canonical bundles of moving frames for these curves in the language of pure Linear Algebra. The scope of applicability of the theory includes geometry of natural classes of curves of flags with respect to
reductive linear groups or their parabolic subgroups. As simplest examples, this includes the projective and affine geometry of curves.
The case of classical groups
is considered in more detail.
\end{abstract}
\maketitle\markboth{Boris Doubrov and Igor Zelenko} {On geometry of
curves of flags of constant type}
\section{Introduction}
\setcounter{equation}{0}
\setcounter{thm}{0}
\setcounter{lem}{0}
\setcounter{prop}{0}
\setcounter{cor}{0}
\setcounter{rem}{0}
\setcounter{exm}{0}
Fix a vector space $W$ over a field $\mathbb K$, where $\mathbb K=\mathbb R$ or $\mathbb C$. Also, fix integers  $0=k_0\leq k_1\leq k_2\leq\ldots\leq k_\mu=\dim W$ and let $F_{k_1,\ldots,k_{\mu-1}}(W)$ be the manifold of all flags $0=\Lambda_0\subset \Lambda_{-1}\subset \Lambda_{-2}\subset\ldots\subset\Lambda_{-\mu}=W$, where $\Lambda_{-i}$ are $k_i$-dimensional linear subspaces.
For definiteness we also assume that $k_1>0$ and $k_{\mu-1}<k_\mu$.
We allow equalities among $k_i$, i.e. repeated subspaces in flags, because direct sums of flags will play an important role in the sequel.

Now fix a Lie subgroup $G$ of $GL(W)$.
The group $GL(W)$  acts naturally on $F_{k_1,\ldots, k_{\mu-1}}(W)$.
Assume that $\mathcal O$ is an orbit in $F_{k_1,\ldots,k_m}(W)$ with respect to the action of $G$.
The general question is whether given two \emph{unparamerized} curves in $\mathcal O$ there exists an element of $G$ sending one curve to another. Such two curves are said to be \emph{$G$-equivalent}. We are also interested in the same question for parametrized curves of flags.

Note that particular examples of this setup include the classical projective and affine geometries of curves in $\mathbb P^n$ and $A^n$ and  the projective geometry of ruled surfaces. In all these cases the action of the group $G$ is transitive on the corresponding flag varieties.
Other examples we will consider in this paper include:
\begin{enumerate}
\item  $G=GL(W)$ and $\mathcal{O}$ is the whole flag variety;
\item $G=Sp(W)$, if $W$ is equipped with a symplectic form $\sigma$,
$G=O(W)$, if $W$ is equipped with a non-degenerate symmetric form $Q$, and $\mathcal{O}$ is the isotropic/coisotropic flag variety.
\end{enumerate}

Our original motivation to study such equivalence problems
comes from the new approach, so-called \emph{symplectification procedure}, to the geometry of structures of nonholonomic nature on manifolds such as vector distributions, sub-Riemannian structure etc. This approach was proposed in \cite{agrgam1, agrachev, jac1, jac2} and it is based on the Optimal Control Theory. It consists of the reduction of the equivalence problem for such nonholonomic geometric structures
to the (extrinsic) differential geometry of curves
in Lagrangian Grassmannians and, more generally,  of curves of  flags of isotropic and coisotropic subspaces in a linear symplectic space with respect to the action of the Linear Symplectic Group (i.e. a particular case of item (2) above). The symplectification procedure was applied to the equivalence problem of vector distributions of rank 2 and 3 (\cite{doubzel1, doubzel2, doubzel3}). For rank 2 distributions curves of flags appearing in this approach are curves of complete flags consisting of all osculating subspaces of the curve of one-dimensional subspaces of these flags, i.e. one arrives to the classical Wilczynski theory of nondegenerate curves in projective spaces \cite{wil}. However,
the geometry of curves of isotropic/coisotropic flags appearing in the symplectification procedure for rank 3 distributions is more involved and needed the development of a new technique.
In \cite{doubzel3}  we treated such curves by a brute force method that cannot be extended to the curves appearing in the theory of distributions of higher rank. The theory developed here gives a conceptual way to work with all such curves.


A general
procedure for the equivalence problems under consideration was developed already by E. Cartan
with his method of moving frames (see \cite{cartan} and modern expositions, for example, by P. Griffiths \cite{griffiths} , M. Green \cite{green}, and M. Fels and P. Olver \cite{olver1}, \cite{olver2}).
In the present paper we distinguish curves of flags for which the construction of canonical bundle of moving frames with respect to the action of a given group $G$ can be done in purely algebraic way and describe this construction in the language of pure Linear algebra.

For a different type of equivalence problems such as equivalence of
filtered structures on manifolds an algebraic version of Cartan's
equivalence method was developed by N. Tanaka in \cite{tan}. Instead
of doing concrete normalizations, Tanaka describes the prolongation procedure for all possible normalizations in purely algebraic terms via so-called universal algebraic prolongation of the symbol of the filtered structure.

We develop a similar algebraic theory for unparametrized curves of flags, satisfying some natural assumptions.
The constructions and the results of the paper can be almost verbatim generalized to embedded submanifolds of flag varieties (see subsection \ref{subman}).
It is worth to notice that  an analog of Tanaka theory for curves (and submanifolds) in projective spaces and more general flag varieties  was already developed in  works of Y. Se-ashi
\cite{seashi1,seashi2}, interpreting geometrically and generalizing the classical work of Wilczynski \cite{wil} (see also \cite{doub3}).
However,  Se-ashi treated much more
restrictive class of equivalence problems compared to our present paper: first, he considers
the case $G=GL(W)$ only and second, he assumes that the algebraic
prolongation of the symbol of the curve is semi-simple. The last
assumption allows him  to associate to a curve of flags (and, more
generally, to a submanifold in a flag variety) a Cartan connection with
values in the algebraic prolongation of the symbol by analogy with
\cite{tan2}.

For the theory of curves (and, more generally, submanifolds) of flags our paper  can be related to Se-ashi works
\cite{seashi1, seashi2}
in the same way
as Tanaka paper \cite{tan} about filtered structures on manifolds with general constant symbol is related to
his later work \cite{tan2} about filtered structures with symbol having semisimple universal algebraic prolongation.

For unparametrized curves in Lagrangian Grassmannians first nontrivial invariants were constructed in the earlier works of the second author with A. Agrachev \cite{jac1, princjac}, using the notion of cross-ration of four points in Lagrangian Grassmannians. Our constructions here give a way to construct a complete system of invariants for curves of flags in much more general situation.

The present paper is closely related to our recent preprint \cite{dzpar} on geometry of curves in parabolic homogeneous space and, more generally, in a homogeneous spaces of a Lie group with a Lie algebra endowed with the fixed grading. The link between these papers is given in Remark \ref{homogenrem}.
In \cite{dzpar} we work with abstract  groups while in the present paper we work with their representations. We had to develop here a special language which is more adapted to the case of representations and to the forthcoming applications to the geometry of
distributions.

The corresponding modification of the theory in the case of \emph{parametrized} curves is given as well (see subsection \ref{param}). This modification give more conceptual point of view on constructions of papers \cite{zelrank1,zelli1,zelli2}  on parametrized curves in Lagrangian Grassmannians and extend them  and results of \cite{jac1, jac2, beffa, beffa1, beffa2, beffa3, ovs} to more general classes of curves.

Let us briefly describe the main constructions and the structure of the present paper.
As in the Tanaka theory for filtered structures on manifolds, the main idea of our approach is to pass from the filtered objects to the corresponding graded objects.
In order to make it work we need additional assumptions on the group $G$ and on the chosen orbit $\mathcal O\subset F_{k_1,\ldots k_m}$ with respect to the action of $G$. This assumptions are discussed in section \ref{comptgrsec} (see Assumption 1 there).
Shortly speaking, any flag $f_0\in\mathcal O$ induces the filtration on the Lie algebra $g$ of the Lie group G. And the compatibility of $\mathcal O$ with respect to the grading means that $g$ is isomorphic (as a filtered Lie algebra) to the associated graded Lie algebra $\gr_{f_0}g$ so that passing to the graded objects we do not change the group in the equivalence problem. Note that $\gr_{f_0}g$ can be identified with a subalgebra of $\gl(\gr_{f_0}W)$, where $\gr_{f_0} W$ is the graded space corresponding to the flag (the filtration) $f_0$. We give an explicit algorithm for constructing of all orbits compatible with respect to the grading under the assumption that $G$ is semisimple (see Proposition \ref{uniflagprop} for the irreducible case) and apply it when $G$ is a symplectic or orthogonal subgroups of $GL(W)$ (see Proposition \ref{compsymp} and Remark \ref{orthrem}, respectively).

The curves of flags under consideration are also not arbitrary but they are compatible  with respect to differentiation (see Assumption 2 in section \ref{symbsec}). Informally speaking, it means that the tangent line to a curve $\tau\mapsto\Lambda(\tau)\subset \mathcal O$ at the point $\Lambda(t)$ is a degree $-1$ element of the graded space $\gr_{\Lambda(t)} g\subset \gl(\gr_{f_0}W)$.  The condition of compatibility with respect to differentiation is natural through the refinement (osculation) procedure on curves of flags described in section \ref{refine}. For example, starting with a curve in Grassmannian, it is natural to produce the curve of flags compatible with respect to differentiation by taking iteratively the osculation subspaces of the original curve. Then the equivalence problem for curves in Grassmannian  is reduced to the geometry of curves in certain flag manifold which are compatible with respect to differentiation. In particular,  in this way geometry of so-called non-degenerate curves in projective space is reduced to geometry of curves of complete flags compatible with respect to differentiation.

Further, in section \ref{symbsec}, similarly to Tanaka theory, we define the symbol of a curve $\tau\mapsto\Lambda(\tau)\subset \mathcal O$ at a point with respect to the group $G$.
For this first we identify the space $\gr_{\Lambda(t)}W$ with a fixed ``model'' graded space $V$ such that this identification conjugate the group $G$ with the fix subgroup $\mathcal G$ of $GL(V)$. Then the tangent line to the curve $\tau\mapsto\Lambda(\tau)$ at the point $\Lambda(t)$ can be identified with a line of degree $-1$ endomorphism in the Lie algebra  $\mathfrak g$ of the Lie group $\mathcal G$. Let $\mathcal G_0$ be the the subgroup of $\mathcal G$ preserving the grading on $V$. Taking all possible  identifications of $\gr_{\Lambda(t)}W$ with $V$ as above we assign to  the tangent line to the curve $\tau\mapsto\Lambda(\tau)$ at $\Lambda(t)$  the orbit of a degree $-1$ endomorphism from $\mathfrak g$ with respect to the adjoint action of
$\mathcal G_0$.
This orbit is called the \emph{symbol of the curve $\tau\mapsto\Lambda(\tau)\subset \mathcal O$ at $\Lambda(t)$ with respect to the group $G$}.

The symbol of the curve of flags at a point is the basic invariant of the curve at this point. The main goal of the present paper is
to study the equivalence problem (w.r.t. to the group $G$) for curves of flags with a given constant symbol. Note that the condition of constancy of symbol is often not restrictive. For example, this is the case, when $G$ is semismple (or, more generally, reductive). It turns out that in this case the set of all possible symbols (for given group $G$) is finite. Therefore,
the symbol of a curve of flags with respect to  a semisimple (reductive) group  $G$  is constant in a neighborhood of a generic point.

The way to solve the equivalence problem under consideration is to associate canonically the bundle of moving frames
to any curve of flags . The main result of the paper (Theorem \ref{mainthm}) shows that in the case of curves with constant symbol the construction of such canonical bundle of moving frames can be done in purely algebraic way, namely in terms of so-called universal algebraic prolongation of the symbol.
The universal algebraic prolongation of the symbol or, more precisely, of the line of degree $-1$ endomorphisms, representing the symbol, is the largest graded subalgebra of $\mathfrak g$ such that its component
corresponding to the negative degrees coincides with this chosen line of degree $-1$ endomorphisms.
It is isomorphic  to the algebra of infinitesimal symmetries of  the so-called flat curve, which  is the simplest (the most symmetric) curve among all curves with this symbol.
In the proof of the main theorem, given in section 5,  we first fix the normalization condition by choosing a complementary subspace to the image of certain coboundary operator in the space of certain 1-cochains. The construction of the bundle of canonical moving frames for any curve with given constant symbol is imitated by  the construction of such bundle for the flat curve with this symbol.

It is important to emphasize that the number of prolongation steps and the dimension of the resulting bundle of moving frame is independent of the choice of the normalization condition but it depends on the symbol only: the number of prolongations steps is equal to the maximal degree in the grading of the universal algebraic prolongation of the symbol and the dimension of the bundle of moving frame is equal to the dimension of the universal algebraic prolongation of the symbol.  The computation of the universal algebraic prolongation is an iterative process, where on each step one needs to solve a system of linear equations. Hence \emph{even without fixing the normalization condition and starting the construction of canonical moving frames one can predict the main features of this construction using linear algebra only}.

Consequently, in order to apply our theory for equivalence of curves of flags with respect to the given group $G$  it is important to classify all possible symbols with respect to this group and to calculate their universal algebraic prolongation. We implement these two tasks in sections \ref{symbclasec} and \ref{algprolongsec}, respectively, for 
the standard representation of classical groups. The universal algebraic prolongation in this cases can be effectively describe using the theory of $\sll_2$-representations.

Our results in the case of Symplectic Group are crucial for application of so-called symplectification procedure to geometry of vector distributions: they give much more conceptual view on our constructions
in \cite{doubzel3} for rank 3 distributions and will be used in our future work on distributions of arbitrary rank. Therefore the  symplectic case is treated in detail.
The case of Orthogonal Group is very similar to the case of Symplectic Group. Hence we will only sketch this case referring to the corresponding objects in the symplectic case.

Note that the set of all possible symbols of curves of flags with respect to a group $G$ depends only on the group $G$ as an abstract Lie group and it does not depend on a particular representation of $G$. Therefore the classification of the symbols in section \ref{symbclasec} and the calculation of section \ref{algprolongsec} can be used for any representation of the classical groups.

In general the bundles of moving frames  obtained in Theorem \ref{mainthm} do not have a structure of a principle bundle. They belong to a wider class of bundles that we call \emph{quasi-principal bundles} (see Definition \ref{constquasidef}). Quasi-principle bundles have some features of the principle bundles when one passes to the grading. The question whether one can choose normalization conditions such that the resulting bundle will be a principle one is reduced to the question whether this normalization condition, as a complementary subspace to the image of certain coboundary operator in the space of certain 1-cochains, is invariant with respect to the natural action of the subgroup of the group of symmetries of the flat curve
preserving a point. In general, such invariant normalization condition may not exist. Some conditions for existence of the invariant normalization conditions and examples of symbols for which they do not exist are given in our recent preprint \cite{dzpar}.

Finally, it is important to stress that we construct canonical bundles of moving frames for curves of flags with a given symbol in a unified way, i.e. without any branching in contrast to our previous construction for curves of isotropic/coisotropic subspaces of a linear symplectic space appearing in the symplectification procedure for rank 3 distributions (\cite{doubzel3}) and also in contrast to the Fels-Olver approach \cite{olver1,olver2}. The latter was used  by G. Mari Beffa (see, for example, \cite{beffa, beffa1, beffa2, beffa3}) for geometry of parametrized curves with very particular symbols in Grassmannians of half-dimensional subspaces
with respect to classical groups
(of Lagrangian subspaces in the symplectic case and of isotropic half-dimensional subspaces in the orthogonal/conformal case).
In the terminology of section \ref{refine}
the first osculating space of such curves at any point is equal to the ambient vector space $W$. The main difference of those works from the treatment of the same curves in the present paper is that in those works not all curves with a given symbol but generic curves are considered. For example, the flat curves do not satisfy the genericity assumptions there.
\medskip

{\bf Acknowledgements} We are very grateful to  professor Pierre Deligne. The idea of treating the equivalence problem for curves of flags by passing to the graded objects stemmed from the way of presentation of the previous paper \cite{zelli1} of the second author with C.Li, which was proposed by professor Deligne during his edition of that paper. Also we would like to thank Professors Tohru Morimoto and Yoshinori Machida for very stimulating discussions.



%

\section{Compatibility of the pair $(G,\mathcal O)$ with respect to grading}
\label{comptgrsec}
\setcounter{equation}{0}
\setcounter{thm}{0}
\setcounter{lem}{0}
\setcounter{prop}{0}
\setcounter{cor}{0}
\setcounter{rem}{0}
\setcounter{exm}{0}
First let us recall some basic notions on filtered and graded vector spaces.
A point $$f_0=\{0=\Lambda_0\subset\Lambda_{-1}\subset \Lambda_{-2}\subset\ldots\subset\Lambda_{-\mu}=W\}$$ of $\mathcal O$ is a decreasing filtration of $W$. So, it induces the decreasing filtration $\{(\gl(W))_{f_0,i}\}_{i\in \mathbb Z}$ of $\gl(W)$,
\begin{equation}
\label{grgldef}
 (\gl(W))_{f_0,i}=\{A\in \gl(W): A(\Lambda_j)\subset \Lambda_{j+i} \text{ for all j}\},\quad (\gl(W))_{f_0,i}\subset (\gl(W))_{f_0,i-1}
\end{equation}
It also induces the filtration on any subspace of $\gl(W)$.
Further, let   $\gr_{f_0} W$ be  the graded
space corresponding to the filtration $f_0$,
$$\gr_{f_0} W=\bigoplus_{i\in\mathbb Z}\Lambda_i/\Lambda_{i+1}$$
and  let $\gr_{f_0}\,\mathfrak {gl}(W)$ be the graded space corresponding to the filtration \eqref{grgldef},
$$\gr_{f_0}\gl(W) =\bigoplus_{i\in\mathbb Z}(\gl(W))_{f_0,i}/(\gl(W))_{f_0,i+1}.$$
The space $\gr_{f_0}\,\mathfrak {gl}(W)$ can be naturally
identified with the space $\mathfrak {gl}\,(\gr_{f_0} W)$. Indeed, if $A_1$ and  $A_2$ from $(\gl(W))_{f_0,i}$ belong to the same coset of
$(\gl(W))_{f_0,i}/ (\gl(W))_{f_0,i+1}$, i.e. $A_2-A_1 \in (\gl(W))_{f_0,i+1}$, and if $w_1$ and $w_2$ from $\Lambda_j$ belong to the same coset of $\Lambda_j/\Lambda_{j+1}$,
i.e. $w_2-w_1\in \Lambda_{j+1}$, then $A_1 w_1$ and $A_2 w_2$ belong to the same coset of $\Lambda_{j+i}/\Lambda_{j+i+1}$.
This defines a linear map from $\gr_{f_0}\,\mathfrak {gl}(W)$ to $\mathfrak {gl}\,(\gr_{f_0} W)$. It is easy to see that this linear map is an isomorphism.

Now let $g\subset \mathfrak {gl}(W)$ be the Lie algebra of the group $G$. The
filtration $f_0$ induces the filtration $\{g_{f_0,i}\}_{i\in\mathbb Z}$ on $g$, where
$$g_{f_0,i}=(\gl(W))_{f_0,i}\cap g.$$ Let $\gr_{f_0}\, g$ be the graded space corresponding to this filtration.
Note that the space $g_{f_0,i}/g_{f_0,i+1}$ is naturally embedded into  the space $(\gl(W))_{f_0,i}/(\gl(W))_{f_0,i+1}$. Therefore,  $\gr_{f_0}\, g$ is naturally embedded into $\gr_{f_0}\gl(W)$ and, by above, $\gr_{f_0}\, g$
can be considered as a subspace of $\mathfrak {gl}\bigl(\gr_{f_0} W\bigr)$. It is easy to see that it
is a subalgebra of $\mathfrak {gl}\bigl(\gr_{f_0} W\bigr)$.

In general, the algebra $\gr_{f_0}\, g$ is not isomorphic to the algebra $g$.
\begin{exm}
\label{nonisograd}
{\rm Assume that $\dim W=4m$, $W$ is equipped with a symplectic (i.e. nondegenerate skew-symmetric) form $\sigma$, $Sp(W)$ is the subgroup of $GL(W)$ preserving the form $\sigma$, and $\mathfrak {sp}(W)$ is the corresponding Lie algebra. Assume that $$f_0=\{0=\Lambda_0\subset \Lambda_{-1}\subset \Lambda_{-2}=W\},$$ where $\dim \Lambda_{-1}=2m$ and the restriction of $\sigma$ to $\Lambda_{-1}$ is nondegenerate. Let us prove that $\gr_{f_0}\, \mathfrak {sp}(W)$ is not isomorphic to $\mathfrak {sp}(W)$ (for more general case see Proposition \ref{compsymp} below). For this first identify the space $W/\Lambda_{-1}$ with the skew-symmetic complement $\Lambda_{-1}^\angle$ of $\Lambda_{-1}$ with respect to the form $\sigma$. Using this identification, we have  that $W/\Lambda_{-1}$ is equipped with the symplectic form, which the restriction of $\sigma$ to   $\Lambda_{-1}^\angle$. Besides $\Lambda_{-1}$ is equipped with the symplectic form, which is the restriction of $\sigma$ to it. Consider the following natural decomposition
\begin{equation}
\label{lambdaWsplit}
\gl(\Lambda_{-1}\oplus W/\Lambda_{-1})=\gl(\Lambda_{-1})\oplus\gl( W/\Lambda_{-1})\oplus \Hom(\Lambda_{-1}, W/\Lambda_{-1})\oplus  \Hom(W/\Lambda_{-1},\Lambda_{-1}).
\end{equation}
Then by direct computations one can show that the algebra $\gr_{f_0}\, \mathfrak {sp}(W)$ is isomorphic to the subalgebra of $\gl(\Lambda_{-1}\oplus W/\Lambda_{-1})$ consisting of endomorphisms $A$ such that if $A$ is decomposed as $A=A_{11}+A_{22}+A_{12}+A_{21}$ with respect to \eqref{lambdaWsplit}
then $A_{11}\in\mathfrak{sp}(\Lambda_{-1})$,  $A_{22}\in\mathfrak{sp}(W/\Lambda_{-1})$, and $A_{21}=0$, where $\mathfrak{sp}(\Lambda_{-1})$ and
$\mathfrak {sp}(W/\Lambda_{-1})$ are symplectic algebras of $\Lambda_{-1}$ and $W/\Lambda_{-1}$, respectively. Consequently, the algebra  $\gr_{f_0}\, \mathfrak {sp}(W)$  is not semisimple and is not isomorphic to  $\mathfrak {sp}(W)$.
More general class of examples is given by Proposition \ref{compsymp} below. On the contrary, if $\Lambda_{-1}$ is a Lagrangian subspace, then similar argument shows that ${\rm gr}_{f_0} \mathfrak{sp}(W)$ is isomorphic to $\mathfrak{sp}(W)$.
$\Box$}
\end{exm}

In order that the passage to the graded objects will not change the group in the equivalence problem we have to impose that $\gr_{f_0}\, g$ and $g$ are conjugated for some (and therefore any) $f_0\in S$. More precisely we will assume in the sequel the following
\medskip

{\bf Assumption 1}\emph{(compatibility with respect to the grading) For some $f_0\in \mathcal O$ , $f_0=\{0=\Lambda_0\subset\Lambda_{-1}\subset \Lambda_{-2}\subset\ldots\subset\Lambda_{-\mu}=W\}$, there exists an isomorphism $J:\gr_{f_0} W\mapsto W$ such that
\begin{enumerate}
\item $J(\Lambda_i/\Lambda_{i+1})\subset \Lambda_i$, $-\mu\leq i\leq -1$;
\item
$J$ conjugates
the Lie algebras $\gr_{f_0}\, g$ and $g$
i.e.
\begin{equation}
\label{gradg}
g=\{J\circ x\circ J^{-1}: x\in \gr_{f_0}\, g\}.
\end{equation}
\end{enumerate}}
\medskip
Note that from the transitivity of the action of $G$ on $\mathcal O$ it follows that if Assumption 1 holds for some $f_0\in\mathcal O$ then it holds for any other $f_0\in\mathcal O$.
If Assumption 1 holds we say that the pair $(G,\mathcal O)$ is \emph{compatible with respect to the grading}.
 Besides, the Lie algebra $g$ has a grading via formula \eqref{gradg} and this grading is defined up to a conjugation.

Obviously, if $G=GL(W)$ or $SL(W)$, then $G$ acts transitively on any flag variety $F_{k_1,\ldots,k_{\mu-1}}(W)$  the pair  $(G,F_{k_1,\ldots,k_{\mu-1}}(W))$ is compatible with respect to the grading.
In general in order to construct a pair $(G, \mathcal O)$ compatible with respect to the grading, one can start with the fixed $\mathbb Z$-grading on the Lie algebra $g$, $g=\displaystyle{\bigoplus_{i\in Z}g_i}$, and
try to find a flag  $f_0$  in $W$ such that the algebra $\gr_{f_0} g$ is conjugated to $g$ and the grading is preserved.
Then as $\mathcal O$ one takes the orbit of $f_0$ with respect to $G$. In this case we say that the orbit $\mathcal O$ is \emph {compatible with respect to the grading of $g$}.

In the case of semisimple $g$ there is an explicit algorithm for constructing all orbits of flags compatible with respect to the grading of $g$.
Recall that an element $e$ of a graded Lie algebra  $g=\displaystyle{\bigoplus_{i\in Z}g_i}$ is called a \emph{grading element}
if
$ad_e(x)=i x$ for any $x\in \mathfrak g_i$.
Since the map $\delta:g\rightarrow g$ sending $x\in g_i$ to $ix$ is a derivation of $g$ and any derivation of a semisimple Lie algebra is inner, for any graded semisimple Lie algebra there exist the unique grading element $e$.
Moreover, this element is also semisimple as an endomorphism of $W$.

\begin{prop}
\label{uniflagprop}
If $G\subset GL(W)$ is a semisimple Lie group acting irreducibly on $W$ and a grading is fixed on its Lie algebra $g$,
then there exists a unique orbit of flags compatible with respect to the grading of $g$.
 \end{prop}

\begin{proof}
Assume that $\lambda$ is the highest weight of the $g$-module $W$ and $\nu$ is the corresponding lowest weight.
Using the basic representation theory of semisimple Lie algebras, one can easily get that in the considered case the spectrum $\text{spec}(e)$ of the grading element $e$
 satisfies
 \begin{equation}
 \label{specE}
 \text{spec}(e)=\{\lambda(e)-i: i\in \mathbb Z, 0\leq i\leq \lambda(e)-\nu(e)\}
 \end{equation}
 (note that $\lambda(e)-\nu(e)$ is natural).
Therefore the natural order on the spectrum of $e$ is defined. Let $W=\displaystyle{\bigoplus_{j=\nu(e)-\lambda(e)-1}^{-1} W_j}$ be the decomposition of $W$ by the eigenspaces of $e$ such that $W_j$ is the eigenspace  corresponding to the eigenvalue $\lambda(e)+j+1$.
Take the flag
\begin{equation}
\label{uniflag}
f_0=\{W^j\}_{j=\nu(e)-\lambda(e)-1}^{-1} \text{ such that } W^j=\displaystyle{\bigoplus_{i\geq j}} W_j.
\end{equation}
If $x\in g_i$ and $w\in W_j$, then $e(w)=(\lambda(e)+j+1)w$ and $[e,x](w)=iw$. Therefore
\begin{equation}
\label{EI}
e\circ x(w)=[e,x](w)+x\circ e(w)=(\lambda(e)+(j+i)+1)x(w),
\end{equation}
i.e. $x(w)\in W_{j+i}$. This implies that the map $J:\gr_{f_0} W\rightarrow W$, which sends an element of $W^j/W^{j+1}$ to its representative in $W_j$, conjugates $\gr_{f_0} g$ and $g$. Therefore the orbit $O_{f_0}$ of $f_0$ with respect to $G$ is compatible with respect to the grading of $g$.

Now let us briefly sketch the proof of uniqueness, which is based on the irreducibility assumption. Since the grading element $e$ belongs to $g_0$, an orbit compatible with respect to the grading of $g$ must contain a flag
$\{\widetilde W^j\}_{j=-\mu}^{-1}$ such that each subspace $\widetilde W_j$ is an invariant subspace of $e$. First one proves that $W^{-1}\subset \widetilde W^{-1}$. Assuming the converse, it is not hard to show that the space $g.\widetilde W^{-1}$ is a proper subspace of $W$, because it does not contain the nonempty set $W^{-1}\backslash \widetilde W^{-1}$. This contradicts the irreducibility assumption. Further, if $W^{-1}$ is a proper subspace of $\widetilde W^{-1}$, then in a similar way one can prove that the space $g.W^{-1}$ does not contain the nonempty set $\widetilde W^{-1}\backslash W^{-1}$ which again contradicts the irreducibility assumption. Hence, $\widetilde W^{-1}=W^{-1}$. In the same manner one can prove that $\widetilde W^{j}=W^{j}$ for any $-\mu\leq j\leq -2$.
\end{proof}

Note that Proposition \ref{uniflagprop} is also true if $G$ is reductive. The proof is the same. The only difference is that the grading element is not unique: it is defined modulo the center of $g$, but it does not effect the proof.


We are especially interested in the case, when $W$ is an even dimensional vector spaces equipped with a symplectic form $\sigma$  and $G=Sp(W)$ or $CSp(W)$, where $Sp(W)$ is the corresponding symplectic group and  $CSp(W)$ is the so-called conformal symplectic group, i.e. the group of all transformation preserving the symplectic form $\sigma$ up to a multiplication by a nonzero constant. Denote by $\mathfrak{sp}(W)$ and $\mathfrak {csp}(W)$ the corresponding Lie algrebras.
Recall that a subspace $L$ of $W$ is called \emph{isotropic} with respect to the symplectic form $\sigma$ if the restriction of the form $\sigma$ on $L$ is identically equal to zero or, equivalently, $L$ is contained in its skew-symmetric complement $L^\angle$ (with respect to $\sigma$), while  a subspace $L$ of $W$ is called \emph{coisotropic} with respect to the symplectic form $\sigma$ if $L$ contains $L^\angle$.

\begin{defn}
\label{sympflagdef}
We say that \emph{a flag $f_0=\{0=\Lambda_0\subset\Lambda_{-1}\subset \Lambda_{-2}\subset\ldots\subset\Lambda_{-\mu}=W\}$ in $W$ is symplectic} if $(\Lambda_{-i})^\angle=\Lambda_{i-\mu}$ for any $0\leq i\leq \mu$.
\end{defn}
Obviously, the flag $f_0$ is symplectic, if and only if
the following  three conditions hold:
any subspace in the flag $f_0$ is either isotropic or coisotropic with respect to the symplectic form $\sigma$;
a subspace belongs to the flag $f_0$ together with its skew-symmetric complement;
the number of appearances of a subspace in the flag $f_0$ is equal to the number of appearances of its skew-symmetric complement in $f_0$.

\begin{prop}
\label{compsymp}
An orbit $\mathcal O$ in a flag manifold is compatible with respect to some grading on $g=\mathfrak {sp}(W)$ or $\mathfrak{csp}(W)$ if and only if some (and therefore any)  flag $f_0\in\mathcal O$ is symplectic.                                                                                                               \end{prop}
\begin{proof}
Let $\lambda$ be the highest weight of the considered standard representation of $\mathfrak {sp}(W)$
Fix the grading on $g$ and let $e$ be the grading element. Since $\mathfrak {csp}(W)=\mathfrak{sp}(W)\oplus\mathbb
K$, in the case $g=\mathfrak{csp}(W)$ we can always choose $e\in \mathfrak{sp}(W)$. By above $\mathcal O$ is the orbit of the flag $f_0$ defined \eqref{uniflag}.
Since $e\in \mathfrak{sp}(W)$, we have that $\sigma(e w_1, w_2)+ \sigma(w_1, e w_2)=0$ for any $w_1$ and $w_2$ in $W$. This implies that if $w_1\in W_{j_1}$, $w_2\in W_{j_2}$, and $\lambda(e)+j_1+1\neq -(\lambda(e)+j_2+1)$ then $\sigma(w_1,w_2)=0$, where $W_j$ is the eigenspace of $e$ corresponding to the eigenvalue   $\lambda(e)+j+1$. From this and the fact that the form $\sigma$ is nondegenerate it follows that the spectrum of $e$ is symmetric with respect to $0$ (for more general statement see Remark \ref{Weylrem} below) and $\sigma$ defines nondegenerate pairing between  $W_{j_1}$ and $W_{j_2}$  with $\lambda(e)+j_1+1= -(\lambda(e)+j_2+1)$. Consequently, the subspaces $W^j$ with $\lambda(e)+j+1>0$ are isotropic and
$W^{j_2}=(W^{j_1})^\angle$ for $\lambda(e)+j_1+1= -(\lambda(e)+j_2+1)$. Thus, the flag $f_0$ is symplectic.
\end{proof}

\begin{rem}
\label{orthrem}
{\rm Assume that $W$ is a vector space equipped with a nondegenerate symmetric form $Q$, and  $G=O(W)$ or $CO(W)$, the orthogonal or conformal groups. The notion of isotropic and coisotropic subspaces of $W$ with respect to the form $Q$  are defined similar to the symplectic case, using orthogonal complements instead of skew-symmetric ones. Then by complete analogy with Proposition \ref{compsymp} the orbits of flags compatible with some  grading of $g=\mathfrak{so}(W)$ or $\mathfrak{cso}(W)$ consist of flags $f_0=\{0=\Lambda_0\subset\Lambda_{-1}\subset \Lambda_{-2}\subset\ldots\subset\Lambda_{-\mu}=W\}$ such that $(\Lambda_{-i})^\perp=\Lambda_{i-\mu}$ for any $0\leq i\leq \mu$ , where $L^\perp$ denotes the orthogonal complement of $L$ with respect to $Q$. We will call such flags \emph{orthogonal}. A flag $f_0$ is orthogonal if and only if the following tree conditions hold: any $\Lambda_i$ is either  isotropic or coisotropic subspaces with respect to $Q$;
a subspace belongs to the flag $f_0$ together with its orthogonal complement with respect to the form $Q$; the number of appearances of a subspace in the flag $f_0$ is equal to the number of appearances of its orthogonal complement in $f_0$.
In particular, if $\mathbb K=\mathbb R$ and the form $Q$ is sign definite then there is no orbits of flags compatible with the grading except the trivial one $0\subset W$.
}$\Box$
\end{rem}

In the case of a general (not necessarily irreducible) representation of a semisimple (a reductive) Lie group G, the flags compatible with the grading can be constructed by the following algorithm:
\begin{enumerate}
\item
take flags as in Proposition 2.1 in each irreducible component;
\item
shift degrees of subspace in each of these flags by arbitrary nonpositive numbers with the only restriction that for at least one irreducible component there is no shift of degrees, i.e. the minimal nontrivial subspace in the flag sitting in this component has degree $-1$.
\end{enumerate}
Consider a flag which is a direct sum of the flags constructed in each irreducible component (with shifted degrees as above): the degree $i$ subspace of this flag is equal to the direct sum of degree $i$ subspaces of flags in each irreducible component.
Obviously, the orbits of such flags are compatible with respect to some grading of $g$. Moreover, they are the only orbits satisfying this property. The above restriction on the shifts of degrees was done in order that the resulting flag will satisfy our convention that the minimal nontrivial subspace in it has degree $-1$.

\begin{rem}
\label{Weylrem}
{\rm  Note that the spectrum of the grading element is symmetric with respect to $0$ for any representation of the following simple Lie algebras  $A_1$, $B_\ell$, $C_\ell$, $D_\ell$ for even $\ell$, $E_7$, $E_8$, $F_4$, and $G_2$ (in  the proof Proposition \ref{compsymp} it was shown for the standard representation of the symplectic Lie algebras $C_\ell$ only).
It follows from the fact that among all simple Lie algebras these are the only algebras for which the map $-1$ belongs to the Weyl group of their root system
(see, for example,\cite[p.71, Exercise 5]{hump}). Thus the highest and the lowest weights $\lambda$ and $\nu$ of any irreducible representation of these algebras satisfy
$\nu=-\lambda$,
which togehter with formula \eqref{specE} implies the desired statement about the spectrum of the grading element.}
$\Box$
\end{rem}

As an example of non-reductive group $G$ consider the case of an \emph{affine subgroup} ${\rm Aff}(W)$ of $GL(W)$ which consists of all elements of $GL(W)$ that preserve  a fixed  affine hyperplane $\mathcal A$ of $W$. Obviously, the restriction of an element of ${\rm Aff}(W)$ to the affine space $\mathcal A$ is an affine transformation of $A$. We define an affine flag in an affine space $\mathcal A$ as a set $\{\mathcal A_i\}_{i=-\mu}^{-1}$ of nested affine subspaces of $\mathcal A$, $\mathcal A_i\subset \mathcal A_{i-1}$. An affine flag is called complete if it consist affine subspaces of all possible dimensions. A flag $\{\Lambda_i\}_{i=-\mu}^{-1}$ in $W$ with $\Lambda_{-1}\cap \mathcal A\neq \emptyset$ defines the  affine flag  $\{\Lambda_i\cap \mathcal A\}_{i=-\mu}^{-1}$ in $\mathcal A$. So, the equivalence problem for curves of flags with respect to the affine group ${\rm Aff}(W)$ can be
reformulated as the equivalence problem for curves of affine flags with respect to the group of affine transformations of $\mathcal A$.
As a particular case we have the classical equivalence problem for curves in an affine space. In order to use our theory for non-degenerate curves in an affine space $\mathcal A$, i.e. curves which do not lie in any proper affine subspace of $\mathcal A$, one has first to make the refinement (osculation) procedure of section \ref{refine} below to reduce the original equivalence problem to the equivalence problem of curves of complete affine flags.
Finally, it is not hard to show that for any orbit $\mathcal O$ of flags with respect to ${\rm Aff}(W)$  the pair  $({\rm Aff}(W),\mathcal O)$ is compatible with respect to the grading.

\begin{rem}
{\rm
The treatment of the reductive case suggests a way for construction of orbits compatible with respect to the grading in the case
of arbitrary (not necessary reductive) graded subalgebra $g\subset \mathfrak{gl}(W)$. Assume that
a grading element $e$ exists and also that $e$, as an endomorphism of $W$, is semisimple.
 Now assume that the spectrum of $e$ is a disjoint union of the sets $\{A_j\}_{j\in \mathbb Z}$ such that if $\lambda$ belongs to $A_j$ and $\lambda+i$ is an eigenvalue of $e$ then $\lambda+i\in A_{j+i}$. As a grading of $W$ take the splitting such that the $j$ subspace of this splitting is the the sum of the eigenspaces of elements  of $A_j$ and as a flag $f_0$ take the corresponding flag as in \eqref{uniflag}. Then by the same arguments as in the proof of Proposition \ref{uniflagprop} we get that the orbit of $f_0$ is compatible with respect to the grading of $g$.}
\end{rem}

\section{Compatibility  with respect to differentiation and symbols of curves of flags}
\setcounter{equation}{0}
\setcounter{thm}{0}
\setcounter{lem}{0}
\setcounter{prop}{0}
\setcounter{cor}{0}
\setcounter{rem}{0}
\setcounter{exm}{0}
\label{symbsec}

After clarifying what kind of orbits in flag varieties will be considered, let us clarify what kind of curves in these orbits will be studied.
Let
\begin{equation}
\label{filt} t \mapsto
\{0=\Lambda_0(t)\subset\Lambda_{-1}(t)\subset\Lambda_{-2}(t)\subset\ldots\subset
\Lambda_{-\mu}(t)=W\}
\end{equation}
be a smooth curve in $\mathcal O$ parametrized somehow.

Recall that for a given parametrization of the curve
 \eqref{filt} the velocity
 $\frac{d}{dt}\Lambda_i(t)$ at $t$ of the curve $\tau\mapsto\Lambda_i(\tau)$ can be naturally identified with an element of ${\rm
 Hom}\bigl(\Lambda_i(t), W/\Lambda_i(t)\bigr)$. Namely, given $l\in \Lambda_i(t)$ take a smooth curve of vector $\ell(\tau)$ satisfying the following two properties:
\begin{enumerate}
\item $\ell(t)=l$,
\item $\ell(\tau)\in\Lambda_i(\tau)$ for any $\tau$ closed to $t$.
\end{enumerate}
Note that the coset of $\ell'(t)$
in $W/\Lambda_i(t)$ is independent of the choice of the curve $\ell$ satisfying the properties (1) and (2) above. Then  to $\frac{d}{dt}\Lambda_i(t)$ we assign the element of
${\rm
 Hom}\bigl(\Lambda_i(t), W/\Lambda_i(t)\bigr)$ which sends $l\in \Lambda_i(t)$ to the coset of $\ell'(t)$
in $W/\Lambda_i(t)$, where the curve $\ell$ satisfies properties (1) and (2) above. It defines a linear map from the tangent space at $\Lambda_i(t)$ to the Grassmannian of $k_{-i}$-dimensional subspace of $W$ to the space  ${\rm
 Hom}\bigl(\Lambda_i(t), W/\Lambda_i(t)\bigr)$. Simple counting of dimensions shows that this map is an isomorphism and therefore it defines the required identification.
 \medskip

{\bf Assumption 2} \emph{(compatibility with respect to differentiation)
We assume that for some (and therefore any) parametrization of the curve
 \eqref{filt}
 the velocity  $\frac{d}{dt}\Lambda_i(t)$ satisfies
 $$\frac{d}{dt}\Lambda_i(t)\in
 {\rm
 Hom}\bigl(\Lambda_i(t),\Lambda_{i-1}(t) /\Lambda_i(t)\bigr), \quad \forall -\mu+1\leq i\leq -1 \text{ and } t.$$}
\medskip

In this case we say that the curve \eqref{filt} is \emph{compatible with respect to differentiation}. Equivalently, a curve \eqref{filt} is compatible with respect to differentiation if
for every $i$, $-\mu+1\leq i\leq -1$, if $\ell(t)$ is a smooth curve
 of vectors such that $\ell(t)\in\Lambda_i(t)$ for any $t$, then
 $\ell'(t)\in \Lambda_{i-1}(t)$ for any $t$.
 The condition of compatibility with respect to differentiation is natural through the refinement procedure on curves of flags described in section \ref{refine} below.

 Under Assumption 2,  $\frac{d}{dt}\Lambda_i(t)$ factors
 through a map $\delta_t$ from $\Lambda_{i}(t) /\Lambda_{i+1}(t)$ to
 $\Lambda_{i-1}(t) /\Lambda_i(t)$.
 In other words,
the map $\delta_t\in \gl\bigl(\gr_{\Lambda(t)} W\bigr)$ of degree $-1$ is well
defined up to a multiplication by a nonzero constant (recall that the reparametrization is allowed). Besides, since the curve \eqref{filt} belongs to the orbit of $G$ then
\begin{equation}
\label{assump3}
\delta_t\in \Bigl(\gr_{\Lambda(t)}\, g\Bigr)_{-1}.
\end{equation}

Now we define symbols of curves of flags at a point with respect to the group $G$.
For this we start with some notations.
Fix $f_0\in \mathcal O$ such that $f_0=\{0=\Lambda_0\subset \Lambda_{-1}\subset \Lambda_{-2}\subset\ldots\subset \Lambda_{-\mu}=W\}$ and denote $V=\gr_{f_0} W$, $V_i=\Lambda_i/\Lambda_{i+1}$. The grading $V=\displaystyle{\bigoplus_{i=-\mu}^{-1}} V_i$  defines  also the natural filtration

\begin{equation}
\label{filtV}
0\subset V^{-1}\subset V^{-2}\subset\ldots\subset V^{-\mu}=V, \quad V^i=\displaystyle{\bigoplus_{j=i}^{-1}} V_j\quad
-\mu\leq i\leq -1.
\end{equation}

Further, fix an isomorphism $J:V\rightarrow W$, satisfying conditions  of Assumption 1. Let $\mathcal G$ be the subgroup of $GL(V)$ such that $\mathcal G=\{J^{-1}\circ A\circ J: A\in G\}$ and $\mathfrak g$ is its Lie algebra. Further, let $\mathfrak {gl}(V)_k$ be the space of endomorphisms of $V$ of degree
$k$, $\mathfrak {gl}(V)_k=\displaystyle{\bigoplus_{i=-\mu}^{-1}{\rm Hom}(V_i, V_{i+k})},$
and
\begin{equation}
\label{degreek}
\mathfrak g_k=\mathfrak g\cap \mathfrak {gl}(V)_k.
\end{equation}

By Assumption 1 this defines the grading of the Lie algebra $\mathfrak g$: $\mathfrak g=\displaystyle{\bigoplus_{k\in Z}}\mathfrak g_k$.
We define a ``big'' bundle $\widehat P$  over
the orbit $\mathcal O$
with the fiber $\widehat P_\Lambda$ over a
point $\widetilde\Lambda=\{\widetilde\Lambda_i\}_{i=-\mu}^{ -1}$ consisting of all isomorphisms $A:V\rightarrow W$
such that
\begin{enumerate}
\item
$A$ preserves the filtrations \eqref{filtV} and $\widetilde\Lambda=\{\widetilde\Lambda_i\}_{i=-\mu}^{ -1}$, i.e. $A(V^i)=\widetilde \Lambda_{i}(t)$
for any $-\mu+1 \leq i \leq
 -1$;
\item
$A$ conjugates
the Lie groups $\mathcal G$ and $G$
i.e.
$G=\{A\circ X\circ A^{-1}: X\in \mathcal G\}$;
.
\item $A\circ J^{-1}\in G$.
\end{enumerate}

Further, let $\mathcal G_+$ be the subgroup of $\mathcal G$ consisting of all elements of $\mathcal G$ preserving the filtration
\eqref{filtV}.
Obviously, the corresponding subalgebra $\mg_+$ satisfies $\mg_+=\displaystyle{\bigoplus_{k\geq 0}\mg_k}$, where $\mg_k$ are as in \eqref{degreek}.
It is easy to see that the bundle $\widehat P$ is a principle $\mathcal G_+$-bundle over
the orbit $\mathcal O$.

\begin{rem}
\label{homogenrem}
(Homogeneous spaces formulation)
{\rm Note that $\widehat P$ can be identified with $G$ via the map $A\mapsto A\circ J^{-1}$ and with $\mathcal G$ via the map $A\mapsto A^{-1} J$, $A\in \widehat P$.
Besides, under the latter identification $\mathcal O\cong\mathcal G/\mathcal G_+$ and the fibers of $\widehat P$ are exactly the left cosets of $\mathcal G_+$ in $\mathcal G$.
The homogeneous space $\mathcal O\cong \mathcal G/\mathcal G_+$ is equipped with the natural $\mathcal G$-invariant vector distribution $D$ which is equal at the ``origin''  of $\mathcal G/\mathcal G_+$ (i.e. at the coset of identity of $\mathcal G/\mathcal G_+$)   to the equivalent classes of elements of degree $-1$  of $\mathfrak g$ under the identification of the tangent space to $\mathcal G/\mathcal G_+$ at the ``origin'' with $\mathfrak g/\mathfrak g_+$. Then a curve in $\mathcal O$ is compatible with respect to differentiation if and only if it is an integral curve of the distribution $D$.
So, \emph{the equivalence problem for curves of flags compatible with respect to differentiation can be reformulated as the equivalence problem for integral curves of the natural $\mathcal G$-invariant distribution in the homogeneous space $\mathcal G/\mathcal G_+$}. This point of view, restricted to the parabolic homogeneous spaces, is considered in our recent preprint \cite{dzpar}
} $\Box$
\end{rem}

Further, the group $\mathcal G_+$ acts naturally on $\mathfrak g_{-1}$  as follows: $A\in \mathcal G_+$ sends $x\in\mg_{-1}$ to the degree $-1$ component of $({\rm Ad} A)\,x$. It induces  the action on the projectivization $\mathbb P\mathfrak g_{-1}$ in the obvious way. By constructions, the set
$\mathfrak M_t=\{\mathbb K \bigl(A^{-1}\circ\delta_t\circ
A\bigr)_{-1}: A\in \widehat P_{\Lambda(t)}\}$,
where
$\bigl(A^{-1}\circ\delta_t\circ
A\bigr)_{-1}$ denotes the degree $-1$ component of $A^{-1}\circ\delta_t\circ
A$
is an orbit
in $\mathbb P\mathfrak g_{-1}$ with respect to
aforementioned action of $\mathcal G_+$. This orbit
is called the \emph{symbol of the curve \eqref{filt} at the point $\{\Lambda_i(t)\}_{i=-1}^{-\mu}$ with respect to $G$}.

\begin{rem}
\label{symbrepind}
{\rm Note that by definition the set of all possible symbols of curves of flags with respect to a group $G$ depends only on the group $G$ as an abstract Lie group. In other words, it does not depend on a particular embedding of $G$ to $GL(W)$ for some $W$ or, equivalently, on a particular representation of $G$. $\Box$}
\end{rem}

In the sequel we will consider curves of flags with the constant symbol $\mathfrak M$, i.e. $\mathfrak M_t=\mathfrak M$ for any $t$.
We also say that such curves of flags are \emph{of constant type $\mathfrak M$}.
If $G$ (and therefore $\mathcal G$) is semisimple and $\mathcal G_0$ is the connected subgroup of $\mathcal G$ with subalgebra $\mg_0$, then due to E.B. Vinberg \cite{vinb} the set of orbits with respect to the adjoint action of $G_0$ on $\mathbb P g_{-1}$ is finite.
Note that if $e$ is the grading element of $\mg$ and $\widetilde{\mathcal G}_0$ is the stabilizer of $e$ with respect to the adjoint action of $\mathcal G$, then the orbits of $\mg_{-1}$ with respect to the natural action of $\mathcal G_{+}$ and the adjoint action of $\widetilde{\mathcal G}_0$ coincide. Besides $\mathcal G_0$ is just the connected component of the identity in $\widetilde {\mathcal G}_0$. Therefore \emph{in the case when $G$ is semisimple the set of all possible symbols is finite} and the condition of constancy of the symbol holds in a neighborhood of a generic point of a curve. The same conclusions can be done if $G$ is reductive.

Any curve of flags that is $G$-equivalent to the curve $t\mapsto \{J\circ e^{t\delta}
V^i\}_{i=0,-1\ldots, -\mu}$ for some $\delta$
such that $\mathbb K\delta \in \mathfrak M$ is called the \emph{flat curve with
constant symbol $\mathfrak M$} (here $V_0=0$). 
The flat curve is an essence the simplest curve among all curves with a given symbol.
\begin{rem}
\label{normalizer}
\rm {If the group $G$ coincides with its normalizer in $GL(W)$, then for all constructions above it is not necessary to fix a map $J:V\mapsto W$ and the condition (3) in the definition of the bundle $\widehat P$ can be omitted.}
\end{rem}

\begin{exm}
\label{projexm1}
{\rm Assume that $\dim W=n+1$ and
consider a curve of complete flags
$$t\mapsto \{0\subset \Lambda_{-1}(t)\subset \Lambda_{-2}(t)\subset \ldots \Lambda_{-(n+1)}(t)=W\},$$
compatible with respect to differentiation. A complete flag means that $\dim \Lambda_{-i}=i$. Assume also that for any $-n<i<-1$
the velocity $\frac{d}{dt}\Lambda_i(t)$ , as an element of $\Hom(\Lambda_i(t),\Lambda_{i-1}(t)/\Lambda_i(t))$, is onto
$\Lambda_{i-1}(t)/\Lambda_i(t)$.
The symbol of such curve of complete flags (with respect to $GL(W)$) is a line of degree $-1$ endomorphisms of the corresponding graded spaces, generated by an endomorphism which has the matrix equal to a Jordan nilpotent block in some basis.
The  flat curve of the maximal refinement is the curve of osculating subspaces of a rational normal curve in the projective space $\mathbb P W$. Recall that a rational normal curve in $\mathbb P W$ is a curve represented as $t\mapsto [1:t:t^2:\ldots t^n]$ in some homogeneous coordinates.}
$\Box$
\end{exm}
As in the Tanaka theory for filtered structures on manifolds, we want to investigate the original equivalence problem via the passage to the graded objects, that is to imitate
the construction of the bundle of canonical moving frames for any curve with a given constant symbol $\mathfrak M$ via the construction of such bundle for the flat curve with symbol $\mathfrak M$.  The latter can be done in purely algebraic way via the notion of the universal algebraic prolongation of the symbol which is introduced in the next section.

\section{Algebraic prolongation of the symbol and the main result}
\setcounter{equation}{0}
\setcounter{thm}{0}
\setcounter{lem}{0}
\setcounter{prop}{0}
\setcounter{cor}{0}
\setcounter{rem}{0}
\setcounter{exm}{0}

From now on we consider a curve of flags \eqref{filt} with constant symbol $\mathfrak M$ and we fix a line $\mathfrak m$ in $\mathfrak g_{-1}$ representing the orbit $\mathfrak M$. Often the line $\mathfrak m$ itself will be called the symbol of the curve \eqref{filt} as well and we will shortly say that the curve \eqref{filt} has constant symbol 
$\mathfrak m$ instead of constant symbol with the representative $\mathfrak m$.
\subsection{Algebraic prolongation of symbol}
\label{univsect}
Set $\mathfrak u_{-1}=\mathfrak m$ and define by induction in $k$
\begin{equation}
\label{kprolong}
\mathfrak u_k:=\{X\in\mathfrak g_k:[X,\delta]\in \mathfrak u_{k-1},\,\delta\in\mathfrak m\},\quad
k\geq 0.
\end{equation}
The space $\mathfrak u_k$ is called the \emph {$k$th algebraic prolongation of the line $\mathfrak m$}.
Then by construction $\mathfrak u(\mathfrak
m)=\displaystyle{\bigoplus_{k\geq -1}\mathfrak u_k}$ is a graded
subalgebra of $\mathfrak g$. It can be shown that it is \emph{the largest
graded subalgebra of $\mathfrak g$ such that its component
corresponding to the negative degrees coincides with
$\mathfrak m$.}
The algebra $\mathfrak u(\mathfrak m)$ is called the \emph{universal algebraic prolongation of the line $\mathfrak m$ (of the symbol $\mathfrak M$)}.
Obviously, $\mathfrak {gl}(V)_k=0$ for all $k\geq \mu$. So $\mathfrak u_k=0$ for $k\geq \mu$.

The algebra $\mathfrak u(\mathfrak m)$ has a very natural geometric meaning. Namely, let $\Gamma$ be
the curve $t\mapsto  \{e^{t\delta}
V^i\}_{i=0,-1,\ldots, -\mu}$  in the corresponding flag manifold $F_{k_1,\dots,k_{\mu-1}}(V)$ for some $\delta\in \mathfrak m$.
The group $GL(V)$ acts naturally on this flag manifold, and thus, we can identify each element  $X\in \mathfrak{gl}(V)$ with a vector field on $F_{k_1,\dots,k_{\mu-1}}(V)$. We define a symmetry algebra of $\Gamma$ as a set of all elements $X\in \mathfrak g$, such that the corresponding vector field on $F_{k_1,\dots,k_{\mu-1}}(V)$ is tangent to $\Gamma$.

As it is shown in \cite{doubkom}, the symmetry algebra of $\Gamma$ is the \emph{largest subalgebra
of $\mathfrak g$ that contains $\mathfrak m$ and lies in $\mathfrak m + \sum_{i\ge 0} \mathfrak g_i$}. It is easy to see from~\eqref{kprolong} that $\mathfrak u$ satisfies this property by construction. Thus, we see that $\mathfrak u(\mathfrak m)$ is exactly a symmetry algebra of the curve $\Gamma$. Besides, a flat curve with symbol $\mathfrak m$ is $G$-equivalent to the curve $J(\Gamma)$. Therefore the algebra $\mathfrak u(\mathfrak m)$ is conjugated to the symmetry algebra of a flat curve with the constant symbol $\mathfrak m$.


\subsection {Zero degree normalization.}
\label{zero}
Let $\widehat P|_{\Lambda(\cdot)}$ be the union of all fibers of the bundle $\widehat P$ over our curve \eqref{filt}, where $\widehat P$ is the bundle over the orbit $\mathcal O$ defined in section \ref{symbsec}.
Our goal is to assign to the curve $\Lambda(\cdot)$ in a canonical way a  fiber subbundle of  $\widehat P|_{\Lambda(\cdot)}$ endowed with a canonical Ehresmann connection, i.e. with a rank 1 distribution transversal to its fibers. We shall construct this bundle through the iterative construction of decreasing sequence of fiber subbundles of $\widehat P|_{\Lambda(\cdot)}$.
Define a map
\begin{equation}
\label{Pkt}
\Pi_k^t:\displaystyle{\bigoplus_{i=-\mu}^{-1}{\rm Hom}\bigl(V_i,\Lambda_i(t)\bigr)}
\rightarrow \displaystyle{\bigoplus_{i=-\mu}^{-1} {\rm Hom}\bigl(V_i,
\Lambda_i(t)/\Lambda_{i+k+1}(t)\bigr)}
\end{equation}
as follows:
 $$\Pi_k^t(A)|_{V_i}\equiv A|_{V_i}\,\, {\rm mod}\,\Lambda_{i+k+1}(t),\quad \forall -\mu\leq i\leq -1.$$

Let $\widehat P_0$ be the subbundle of $\widehat P|_{\Lambda(\cdot)}$
 with the fiber $\widehat P_0(t)$ over the point $\{\Lambda_i(t)\}_{i=-\mu}^{ -1}$ consisting of all $A\in \widehat P_{\Lambda(t)}$
such that
 $$\mathfrak m=\mathbb K \bigl(A^{-1}\circ\delta_t\circ
A\bigr)_{-1}$$

We also denote by $P_0$ the bundle over our curve with the fiber $P_0(t)$ over the point $\{\Lambda_i(t)\}_{i=-\mu}^{ -1}$ equal to the image of the  corresponding fiber $\widehat P_0(t)$ of the bundle $\widehat P_0$ under the map $\Pi_0^t$.
 By constructions $P_0$  is a principal $U_0$-bundle, where $U_0$ is a subgroup of $\mathcal G$ consisting of all degree $0$ elements $B$ of $\mathcal G$ such that ${\rm Ad} B(\mathfrak m)=\mathfrak m$. The Lie algebra of $U_0$ is equal to $\mathfrak u_0$ defined by \eqref{kprolong}.

\subsection{Quasi-principal subbundle of $\widehat P_0$}
\label{quasi}
Take a fiber subbundle $P$ of $\widehat P_0$ which is not necessary a principle subbundle of $\widehat P_0$.
Let $P(t)$ be the fiber of $P$ over the point $\{\Lambda_i(t)\}_{i=-1}^{-\mu}$.
Take $\psi \in P(t)$.
The tangent space $T_\psi\bigl(P(t)\bigr)$ to the fiber $P(t)$ at
a point $\psi$ can be identified with a subspace of $\gl(V)$.
Indeed, define the following $\mathfrak g_+$-valued $1$-form $\omega$ on $P$: to any vector $X$ belonging to
$T_\psi\bigl(P(t)\bigr)$ we assign an element $\omega(\psi)(X)$ of
$\mathfrak g_+$
as follows: if $s\to \psi(s)$ is a smooth curve in $P(t)$ such that $\psi(0)=\psi$ and
$\psi'(0)=X$ then let
\begin{equation}
\label{omega}
\omega(\psi)(X)=\psi^{-1}\circ X,
\end{equation}
where in the last
formula by $\psi$ we mean the isomorphism
between $V$ and $W$.
Note that the linear map $\omega(\psi): T_\psi\bigl(P(t)\bigr)\mapsto \mathfrak g_+$ is injective.
Set

\begin{equation}
\label{Lpsi}
L_\psi:=\omega(T_\psi\bigl(P(t)\bigr).
\end{equation}

If $P$ is a principle bundle over our curve, which is a reduction of the bundle $\widehat P_0$, then the space $L_\psi$ is independent of $\psi$ and equal to the Lie algebra of the structure group of the bundle $P$. For our purposes here we need to consider more general class of fiber subbundles of $\widehat P_0$. To define this class recall that the  filtration given by the subspaces $V^i$, $-\mu\leq i\leq -1$, defined by
\eqref{filtV}, induces a natural filtration on $\gl(V)$
and, therefore, on the subspace $L_\psi$. The corresponding graded
subspace $L_\psi$ is called a \emph{symbol of the bundle $P$
at a point $\psi$}. Under the natural identification of spaces $\gr \gl(V)$ with $\gl(\gr V)$, described in the beginning of section
\ref{comptgrsec},
one has that $\gr L_\psi$ is a subspace of $\gl( \gr V)$.
Besides $V$ is a graded space by definition, i.e $V\sim \gr V$.
Therefore, $\gr L_\psi$ is a subspace of $\gl(V)$.

\begin{defn}
\label{constquasidef}
We say that the bundle $P$
has a \emph{constant symbol $\mathfrak s$} if its symbols at different points
coincide with $\mathfrak s$. In this case we call $P$ the \emph{quasi-principle subbundle of the bundle $\widehat P_0$ with symbol $\mathfrak s$.}
\end{defn}

\subsection{Structure function associated with Ehresmann connection}
\label{structsec}
Further, assume that a fiber subbundle $P$
of $\widehat P_0$ is endowed with an Ehresmann connection, i.e. with a rank 1 distribution $\mathcal H$ transversal to its fibers. 
A parametrized smooth curve $t\mapsto \psi(t)$ in $P$ is called a \emph{ moving frame of the pair $(P, \mathcal H)$} if the following two conditions holds:
\begin{enumerate}
\item
 the curve $t\mapsto \psi(t)$ is tangent  to the distribution $\mathcal H$ at any point;
\item
there exists $\delta\in\mathfrak m$ such that  $\bigl(\psi^{-1}\circ \psi'(t)\bigr)_{-1}=\delta$ for any $t$.
\end{enumerate}
A pair $(P,\mathcal H)$ will be called a \emph{bundle of moving frames}.

Let $C^1={\rm Hom}(\mathfrak m, \mathfrak g)$.
Also, given $A\in \mathfrak{gl}(V)$ denote by $(A)_k$ the component of degree $k$ of $A$ (w.r.t. the splitting
$\mathfrak{gl}(V)={\displaystyle\bigoplus_{k\in\mathbb Z}\mathfrak{gl}(V)_k}$).
 Then to any bundle $P$ with the fixed Ehresmann connection $\mathcal H$ one can assign the function $c:P\rightarrow C^1$ as follows: Given  $\psi\in P(t)$ and $\delta \in \mathfrak m$  let $\psi(\tau)$ be a curve in $P$ tangent to the Ehresmann connection $\mathcal H$ such that
$\psi(t)=\psi$ and $\bigl(\psi^{-1}\circ \psi'(t)\bigr)_{-1}=\delta$.
Then set

\begin{equation}
\label{structeq}
c(\psi)(\delta):=\psi^{-1}(t)\circ \psi'(t).
\end{equation}
Note that the righthand side of \eqref{structeq} does not depend on the choice of the curve $\psi(tau)$ with the aforementioned properties.
The function $c$ is called the \emph{structure function of the pair the bundle of moving frames $(P,\mathcal H)$}.
\begin{rem}
\label{Maurer}
\rm{Under the identification of the bundle $\widehat P$ with the Lie group $\mathcal G$ given in Remark \ref{homogenrem} one can describe the structure function $c(\psi)$ in terms of the \emph{left-invariant Maurer-Cartan form $\Omega$} of $\mathcal G$ by the following formula:
\begin{equation}
\label{structMaurer}
c(\psi)\Bigl(\bigl(\Omega(X)\bigr)_{-1}\Bigr)=\Omega(\psi)(X)\quad  \forall \psi\in P \text{ and }  X\in \mathcal H(\psi),
\end{equation}
where $\bigl(\Omega(X)\bigr)_{-1}$ is the degree $-1$ component of $\Omega(X)$ with respect to the grading on $\mathfrak g$.
Note also that the
$\mathfrak g_+$-valued $1$-form $\omega$ defined by \eqref{omega} is nothing but the restriction of the left-invariant Maurer-Cartan form to the fibers of the bundle $P$.} $\Box$
\end{rem}

Note that
\begin{equation}
\label{splitting}
C^1=\displaystyle{\bigoplus_{k\in \mathbb Z} C^1_k},
\end{equation}
where $C^1_k={\rm Hom}(\mathfrak m, \mathfrak \mg_{k-1})$.
Let $C^1_+:= \displaystyle{\bigoplus_{k>0} C^1_k}$, $c_k$ be the $k$th component of $c$ w.r.t. the splitting \eqref{splitting},
and $c_+=\sum_{k>0 0} c_k$. We say that  $c_+$ is  the \emph{positive part of the structure function of $c$}.
Note that by the constructions $c_k=0$ for $k<0$ and  $c_{0}(\delta)=\delta$ for every
$\delta\in \mathfrak m$.

Further, let $\partial:\mathfrak g\rightarrow C^1$ be the operator given  by the following formula:

\begin{equation}
\label{coboundary}
\partial x(\delta)=[\delta, x],\quad \forall x\in \mathfrak g, \delta\in \mathfrak m.
\end{equation}

\begin{rem}
\label{cohomrem}
\rm
{As a matter of fact, one can look on $\mathfrak g$ and $C^1$ as on the spaces of $0$-cochains and $1$-cochains, respectively, on $\mathfrak m$  with values in $\mathfrak g$. Moreover, the operator $\partial$ is exactly the coboundary operator associated with the
adjoint representation ${\rm ad}:\mathfrak m\rightarrow \mathfrak{gl}(\mg)$.}
\end{rem}

Recall also that the group $\mathcal G_+$ acts naturally on the space $C^1$ as follows:
\begin{equation}
\label{action}
(A.c)(\delta)=({\rm Ad} A) c\Bigl(\bigl(({\rm Ad} A^{-1})\delta\bigr)_{-1}\Bigr), \quad a\in\mathcal G_+, \,\delta\in \mathfrak m,
\end{equation}
where $\bigl(({\rm Ad} A^{-1})\delta\bigr)_{-1}$ is a degree $-1$ component of $({\rm Ad} A^{-1})\delta$ with respect to the grading of $\mg$.
Obviously, this action restricts to the action on $C^1_+$.

Finally, let $\mathfrak u_+(\mathfrak
m)=\displaystyle{\bigoplus_{k\geq 0}\mathfrak u_k}$ and let $U_+
(\mathfrak m)
$ be the subgroup of the group of symmetries of the curve $t\mapsto  \{e^{t\delta}
V^i\}_{i=-1,\ldots, -\mu}$ preserving the filtration  $\{V^i\}_{i=-\mu}^{-1}$. Note that  $\mathfrak u_+(\mathfrak m)$ is the Lie algebra of $U_+(\mathfrak m)$.

A subspace $\mathcal N$ complimentary to the space ${\rm Hom} (\mathfrak m, \mathfrak u_+(\mathfrak m))+({\rm Im}\,\partial\cap C^1_+)$ in $C^1_+$ is called a \emph{normalization condition}.
We say that \emph{an Ehresmann connection $\mathcal H$ on a subundle $P$ of $\widetilde P_0$  satisfies the normalization condition $\mathcal N$} if the positive part of the  structure function of the pair $(\mathcal P,\mathcal H)$ takes values in $\mathcal N$.

Our main theorem can be formulated as follows:
\begin{thm}
\label{mainthm}
Fix a normalization condition $\mathcal N$.
Then for a curve of flags \eqref{filt} with constant symbol with representative $\mathfrak m$
there exists a unique quasi-principal subbundle $P$
of the bundle $\widehat P_0$ with symbol $\mathfrak u_+(\mathfrak
m)$  and a unique Ehresmann connection $\mathcal H$ on it such that this connection satisfies the normalization conditions $\mathcal N$.
If, in addition,  the space $\mathcal N$ is invariant with respect to the natural action of the subgroup $U_+(\mathfrak m)$ on $C^1_+$,
then the bundle $P$ is the principal bundle with the structure group  $U_+(\mathfrak m)$.
\end{thm}
The pair $(P, \mathcal H)$ from the previous theorem is called \emph{the bundle of moving frames canonically associated with the curve \eqref{filt} via the normalization condition $\mathcal N$} and the Ehresmann connection $\mathcal H$ is called  \emph{the canonical Ehresmann connecton on $P$ associated with the curve \eqref{filt} via the normalization condition $\mathcal N$}.

Theorem \ref{mainthm} is proved in the next section.
As a direct consequence, we have that two curves $\{\Lambda_i\}_{i=-\mu}^{-1}$ and $\{\widetilde \Lambda_i\}_{i=-\mu}^{-1}$ of flags with constant symbol $\mathfrak m$ are equivalent through $A\in G$ if and only if for any moving frame $t\mapsto \psi(t)$ of the pair (the bundle $P$, the Ehresmann connection $\mathcal H$)  canonically associated with the curve $\{\Lambda_i\}_{i=-\mu}^{-1}$ via $\mathcal N$ the curve  $t\mapsto A\circ\psi(t)$ is the moving frame of the pair (the bundle $\widetilde P$, the Ehresmann connection $\widetilde{\mathcal H}$) canonically associated with the curve $\{\widetilde\Lambda_i\}_{i=-\mu}^{-1}$ via $\mathcal N$.

\begin{rem}
\label{princrem}
{\rm Note that in the case when the normalization condition $\mathcal N$ is invariant with respect to the natural action of the subgroup $U_+(\mathfrak m)$ on $C^1_+$, the canonical Ehresmann connection is not a principal connection on the corresponding
$U_+(\mathfrak m)-$ principal bundle in general. It is a principal connection if the first algebraic prolongation $\mathfrak u_1$ is equal to $0$.}
\end{rem}


%

\subsection {On geometry of parametrized curves of flags.}
\label{param}
 Similar results can be obtained for equivalence problem for \emph{parametrized} curves of flags with respect to the action of group $G$. There are only three modifications. The first modification is that the symbol of a parametrized  curve is an orbit of an element (and not a line) in $\mathfrak g_{-1}$ with respect to the adjoint action of $\mathcal G_+$. Then we fix a representative $\delta$
of this orbit. The second modification is in the definition of the algebra $\mathfrak u_0$ and of the bundle $\widehat P_0$. Let $\delta\in\mathfrak g_{-1}$.
Set $\mathfrak u_{-1}=\mathbb K\delta$ and
\begin{equation}
\label{prolong0par}
\mathfrak u_0:=\{X\in\mathfrak g_0:[X,\delta]=0\}.
\end{equation}
The spaces $\mathfrak u_k$ for $k>0$ are defined recursively, using \eqref{kprolong}.
The third modification is in the definition of the bundle $\widehat P_0$ (the zero degree reduction of the bundle $P|_{\Lambda(\cdot)}$). Here $\widehat P_0$ should be the subbundle of $\widehat P|_{\Lambda(\cdot)}$
 with the fiber $\widehat P_0(t)$ over the point $\{\Lambda_i(t)\}_{i=-\mu}^{ -1}$ consisting of all $A\in \widehat P_{\Lambda(t)}$
such that
$$\mathfrak \delta=\bigl(A^{-1}\circ\delta_t\circ
A\bigr)_{-1}.$$

With these modifications all other constructions and Theorem \ref{mainthm} are valid.
\subsection{On geometry of submanifolds of constant type in flag varieties}
\label{subman}
Similar theory can be constructed for equivalence problem for submanifolds of dimension $l$ of $\mathcal O$ with respect to the action of the group $G$.
Again we consider submanifolds compatible with respect to differentiation. A submanifold is compatible with respect to differentiation if any smooth curve on it is compatible with respect to differentiation.  A symbol of an $l$-dimensional submanifold $S\subset\mathcal O$ at a point $\Lambda$ is an orbit of an $l$-dimensional abelian subalgebra $\mathfrak m$ (of $\mg$) belonging to $\mg_{-1}$ with respect to the natural action of $\mathcal G_+$ on $l$-dimensional subspaces of of $\mg_{-1}$. This subalgebra can be taken as a linear span of all elements of type $\bigl(A^{-1}\circ\delta_t\circ
A\bigr)_{-1}$, where $A$ is a fixed element of the fiber $\widetilde P_\Lambda$ and $\delta_t$ are elements of $\gr_\Lambda W$ corresponding to all possible smooth curves in $S$, passing through $\Lambda$. Note that the word "abelian" can be removed, because any subalgebra of $\mg$ that belongs to $\mg_{-1}$ must be abelian.

Further,  one can define the bundle $\widehat P_0$ over the submanifold $S$ as in subsection \ref{zero}, the notion of a quasi-principal subbundle of the bundle $\widehat P_0$ as in subsection \ref{quasi}, an Ehresmann connection of $P$ i.e. a rank $l$ distribution $\mathcal H$ transversal to its fibers, and a moving frame of the pair $(P,\mathcal H)$ as in subsection \ref{structsec}.
The structure function of the  $(P,\mathcal H)$  can be defined by the relation \eqref{structMaurer}. Note that since $\mathfrak m$ is abelian, the structure function  at any point $\psi\in P$ takes its values in the space $Z^1$ of $1$-cocycles (with respect to the coboundary operator associated with the
adjoint representation ${\rm ad}:\mathfrak m\rightarrow \mathfrak{gl}(\mg)$, see Remark \ref{cohomrem}). Therefore, to get the generalization of Theorem \ref{mainthm} to submanifolds with a constant symbol in flag varieties (with literally the same proof as in section \ref{proofsec} below)  one has to replace $C^1_+$ by $Z^1_+$ in the formulation of this theorem, where $Z^1_+$ is the subspace of cocycles in $C^1_+$. It is clear that in the case of curves
$Z^1=C^1$.

Finally note that one can consider  more general situation of filtered submanifolds of constant type in flag variety, which is encoded by
fixing a graded subalgebra in a negative part of $\mathfrak g$ (not necessary belonging to $\mathfrak g_-1$. This point of view is considered in the recent preprint \cite{doubmor}.

\section{Proof of Theorem \ref{mainthm}}
\label{proofsec}
\setcounter{equation}{0}
\setcounter{thm}{0}
\setcounter{lem}{0}
\setcounter{prop}{0}
\setcounter{cor}{0}
\setcounter{rem}{0}
\setcounter{exm}{0}

We construct  a decreasing sequence of fiber subbundles of the bundle $\widehat P_0$ by induction.



First let us introduce some notations.
If $\mathcal N$ is the space  from the formulation of Theorem \ref{mainthm}, denote
\begin{equation}
\label{Nk}
\mathcal N_k=\mathcal N\cap C^1_k,\quad k\in\mathbb N
\end{equation}
Then obviously $\mathcal N=\displaystyle{\bigoplus_{k\in\mathbb N}\mathcal N_k}$.

Note also that there is a natural mappings
${\rm pr}_k: C^1_k\mathfrak\rightarrow {\rm Hom}(\mathfrak m, \mathfrak g_{k-1}/\mathfrak u_{k-1})$ such that for any $c_k\in C^1_k$ and $\delta\in\mathfrak m$ we set ${\rm pr}_k(c_k)(\delta)$ to be the equivalence class of $c_k(\delta)$ in  $\mathfrak g_{k-1}/\mathfrak u_{k-1}$. We assume that
\begin{equation}
\label{barNk}
\overline{\mathcal N}_k={\rm pr}_k(\mathcal N_k).
\end{equation}

Further, let $\mathfrak g_k$ be as in \eqref{degreek}. Then the image of the restriction of the operator $\partial$ to the subspace $\mathfrak g_k$ belongs to $C^1_{k}$. We denote this restriction by $\partial_k$, $\partial_k: \mathfrak g_k\rightarrow C^1_{k}$. Finally, let
\begin{equation}
\label{barpartial}
\bar\partial_k={\rm pr}_{k}\circ\partial_k.
\end{equation}
Note that by constructions the space $\mathfrak u_k$ defined by \eqref{kprolong} is nothing but the kernel of the map $\bar\partial_k$,
$\mathfrak u_k=\ker \bar\partial_k$.

\subsection {First degree normalization.}
First take some Ehresmann connection on the bundle $\widehat P_0$ and let $c_1$ be the degree $1$ component of the structure function of the bundle $\widehat P_0$ with this Ehresmann connection. If one takes another Ehresmann connection on $\widehat P_0$ and the corresponding degree $1$ component $\tilde c_1$ of the structure function of $\widehat P_0$ with this new Ehresmann connection, then

\begin{equation}
\label{free1}
 \tilde c_1(\psi)(\delta)-c_1(\psi)(\delta)\in\mathfrak u_0 \quad \forall \psi\in \widehat P_0 \text{ and } \delta\in \mathfrak m
 \end{equation}
 In other words, the map $\bar c_1:\widehat P_0\rightarrow {\rm Hom}(\mathfrak m, \mathfrak g_0/\mathfrak u_0)$,
defined by  $\bar c_1(\psi)={\rm pr}_1\circ c_1 (\psi)$,
is independent  of a choice of an Ehresmann connection on $\widetilde P_0$.

Now take  $\psi_1$ and $\psi_2$ from the same fiber $\widehat P_0(t)$  such that $\Pi_0^t(\psi_1)=\Pi_0^t(\psi_2)$ and
study how $\bar c_1(\psi_1)$ and $\bar c_1(\psi_2)$ are related.
As before given $A\in \mathfrak{gl}(V)$ denote by $(A)_k$ the component of degree $k$ of $A$ (w.r.t. the splitting
$\mathfrak{gl}(V)={\displaystyle\bigoplus_{k\in\mathbb Z}\mathfrak{gl}(V)_k}$). Set
\begin{equation}
\label{f}
f_{\psi_1,\psi_2}^1:=\bigl(\psi_1^{-1}\circ \psi_2\bigr)_1.
\end{equation}

It is not hard to see that  $f_{\hat\vf,\hat\vf_1}\in\mathfrak g_1$, where $\mathfrak g_1$ is defined by \eqref{degreek}.
In opposite direction, for fixed $\psi_1 \in \widehat P_0(t)$ and any $f\in \mathfrak g_1$ there exists   $\psi_2\in\widehat P_0(t)$ such that $\Pi_0^t(\psi_1)=\Pi_0^t(\psi_2)$ and $f=f_{\psi_1,\psi_2}^1$.

\begin{lem}
\label{lem0}
The following identity holds
\begin{equation}
\label{structrans}
\bar c_1(\psi_2)=\bar c_1(\psi_1)+\bar\partial_1 (f_{\psi_1,\psi_2}^1).
\end{equation}
\end{lem}

\begin{proof}
First note that

\begin{equation}
\label{fhalf}
f_{\psi_1,\psi_2}^1=\bigl(\Pi_0^t(\psi_1)\bigr)^{-1}\circ \bigl(\Pi_1^t(\psi_2)-\Pi_1^t(\psi_1)\bigr).
\end{equation}
or, equivalently,
\begin{equation}
\label{f1}
\psi_2(v)=\psi_1(v)+\Pi_0^t(\psi_1)\circ f_{\psi_1,\psi_2}(v)\,\,{\rm mod}\, \Lambda_{i+2}(t) \quad \forall v\in V_i.
\end{equation}

Let $\psi_i(\tau)$, $i=1,2$ be a section of $\widehat P_0$  such that $\psi_i(t)=\psi_i$ and
\begin{equation}
\label{psidelta}
\bigl(\psi_i^{-1}\circ \psi_i'(t)\bigr)_{-1}=\delta
\end{equation}
Then one can write the identity analogous to \eqref{f1} for any $\tau$. Differentiating it at $\tau=t$ and using \eqref{psidelta} again we get
\begin{equation}
\label{f2}
\psi_2'(t)(v)=\psi_1'(t)(v)+\Pi_0^t(\psi_1)\circ\delta\circ f^1_{\psi_1,\psi_2}(v)\,\,{\rm mod}\, \Lambda_{i+1}(t) \quad \forall v\in V_i.
\end{equation}
Applying $\psi_2^{-1}$ to both sides and using \eqref{psidelta} again one gets
\begin{equation}
\label{f3}
\psi_2^{-1}\circ\psi_2'(t)(v)=\psi_2^{-1}\circ\psi_1'(t)(v)+\delta\circ f^1_{\psi_1,\psi_2}(v)\,\,{\rm mod}\, \Lambda_{i+1}(t) \quad \forall v\in V_i.
\end{equation}
Representing $\psi_2^{-1}\circ\psi_1'(t)$ as $\psi_2^{-1}\circ \psi_1\circ \psi_1^{-1}\circ\psi_1'(t)$ and taking into account that
\begin{equation*}
\begin{split}
~&\psi_2^{-1}\circ \psi_1=Id-f^1_{\psi_1,\psi_2}+\text{operators of degree $\geq 2$},\\
~&\psi_1^{-1}\circ\psi_1'(t)\equiv \delta+\bar c_1(\psi_1)+\text{operators of degree $\geq 1$}\quad {\rm mod}\,\, \mathfrak u_0.
\end{split}
\end{equation*}
we get that
$$\bigl(\psi_2^{-1}\circ\psi_1'(t)\Bigr)_0=\bar c_1(\psi_1)-f^1_{\psi_1,\psi_2}\circ \delta \quad {\rm mod}\,\, \mathfrak u_0.$$
Finally, comparing $0$-degree components of identity \eqref{f3} we get the required identity  \eqref{structrans}.
%
%
\end{proof}

Now take the space $\mathcal N_1$ as in \eqref{Nk}. By our assumptions on $\mathcal N$ we have

\begin{equation}
\label{normsplit}
{\rm Hom}(\mathfrak m, \mathfrak g_0/\mathfrak u_0)= \text{Im} \,\bar\partial_1\oplus\overline {\mathcal N}_1.
 \end{equation}

 Let $\widehat P_1$ be the subbundle of the bundle $\widehat P_0$ consisting of all $\psi\in \widehat P_0$ such that
 $\bar c_1(\psi)\in \overline{\mathcal N}_1$. From formula \eqref{structrans} and splitting \eqref{normsplit} it follows that the bundle $\widehat P_1$ is not empty. Moreover, $\psi_1$ and $\psi_2$ from the same fiber $\widehat P_1(t)$  satisfy $\Pi_0^t(\psi_1)=\Pi_0^t(\psi_2)$ if and only if $f_{\psi_1, \psi_2}^1\in \ker \bar\partial_1=\mathfrak u_1$.

 We also denote by $P_1$ the bundle over our curve with the fiber $P_1(t)$ over the point $\{\Lambda_i(t)\}_{i=-\mu}^{ -1}$ equal to the image of the  corresponding fiber $\widehat P_1(t)$ of the bundle $\widehat P_1$ under the map $\Pi_1^t$. By above $P_1$ can be considered as an affine bundle over $P_0$ such that its fibers are affine spaces over the vector space $\mathfrak u_1$. The projection of the bundle is induced by the maps $\Pi_0^t$.

 Note that in the constructions of the bundles $\widehat P_1$ and $P_1$ we did not use the whole information on the space $\mathcal N_1$, but on its image under the (non-injective) map ${\rm pr}_1$ only. So, from the knowledge of the whole space $\mathcal N_1$ we should get additional restrictions on the bundle of moving frame we are looking for.
Indeed, by our assumptions on $\mathcal N$ we have the following splitting
\begin{equation}
({\rm pr}_1)^{-1}(\overline{\mathcal N}_1)=\mathcal N_1\oplus{\rm Hom}(\mathfrak m, \mathfrak u_0).
\end{equation}
From this splitting and relation \eqref{free1} it follows that one can choose an Ehresmann connection  on $\widehat P_1$ such that
the degree $1$ component $c_1$ of the structure equation of $\widehat P_1$ with this Ehresmann connection belongs to $\mathcal N_1$. Moreover, any such Ehresmann connection on $\widehat P_1$ induces the unique Ehresmann connection on the bundle $P_0$.

\subsection {Higher degree normalizations: the induction step}
 Given $i>0$ assume that there exists a unique series of bundles $\widehat P_k$, $1\leq k\leq i$ over our curve such that $\widehat P_0$ is as above and for each $1\leq k\leq i$ the bundle $\widehat P_k$ satisfies the following properties:
 \begin{enumerate}
 \item [{\bf (A1)}]$\widehat P_k$ is a subbundle of $\widehat P_{k-1}$;
 \item [{\bf (A2)}]if $c$ is the structure function of $\widehat P^k$ (endowed with some Ehresmann connection)  and $c_l$ is its degree $l$ component, then
\begin{equation}
\label{widePk}
{\rm pr}_l\circ c_l \in\overline{\mathcal N}_l \quad \forall\, 1\leq l\leq k
\end{equation}
\item [{\bf (A3)}] if $P_k$, $0\leq k\leq i$,  is the bundle over our curve with the fiber $P_k(t)$ over the point $\{\Lambda_j(t)\}_{j=-\mu}^{ -1}$ equal to the image of the  corresponding fiber $\widehat P_k(t)$ of the bundle $\widehat P_k$ under the map $\Pi_k^t$, then $P_k$ can be considered as an affine bundle over $P_{k-1}$ such that its fibers are affine spaces over the vector space $\mathfrak u_k$. The projection of this bundle is induced by the maps $\Pi_{k-1}^t$;
\item [{\bf (A4)}]   there exists an Ehresmann connection on the bundle $\widehat P_k$ such that the structure function $c$ satisfies
\begin{equation}
\label{Ehresk}
c_l \in\mathcal N_l\quad \forall\, 1\leq l\leq k.
\end{equation}
Moreover, any such Ehresmann connection on $\widehat P_k$ induces the unique Ehresmann connection on the bundle $P_{k-1}$.
\end{enumerate}

Now let us construct the subbundle $\widehat P_{i+1}$ of $\widehat P_i$, corresponding to the normalization of degree $i+1$.
For this first take some Ehresmann connection on $\widehat P_i$ satisfying \eqref{Ehresk} (with $k=i$) and let $c_{i+1}$ be the degree $i+1$ component of the structure function of the bundle $\widehat P_i$ with this Ehresmann connection. If one takes another Ehresmann connection on $\widehat P_i$ satisfying \eqref{Ehresk} (with $k=i$) and the corresponding degree $i+1$ component $\tilde c_{i+1}$ of the structure function of $\widehat P_i$ with this new Ehresmann connection, then

\begin{equation}
\label{freei}
 \tilde c_{i+1}(\psi)(\delta)-c_{i+1}(\psi)(\delta)\in\mathfrak u_i \quad \forall \psi\in \widehat P_i \text{ and } \delta\in \mathfrak m.
 \end{equation}
 In other words, the map $\bar c_{i+1}:\widehat P_i\rightarrow {\rm Hom}(\mathfrak m, \mathfrak g_i/\mathfrak u_i)$,
defined by  $\bar c_{i+1}(\psi)={\rm pr}_{i+1}\circ c_{i+1} (\psi)$,
is independent  of a choice of an Ehresmann connection on $\widetilde P_i$.

Now take  $\psi_1$ and $\psi_2$ from the same fiber $\widehat P_i(t)$ of the bundle $P_i$ such that $\Pi_i^t(\psi_1)=\Pi_i^t(\psi_2)$ and
study how $\bar c_{i+1}(\psi_1)$ and $\bar c_{i+1}(\psi_2)$ are related.
Set
\begin{equation}
\label{fi}
f_{\psi_1,\psi_2}^{i+1}:=\bigl(\psi_1^{-1}\circ \psi_2\bigr)_{i+1}.
\end{equation}

It is not hard to see that  $f_{\hat\vf,\hat\vf_{i+1}}\in\mathfrak g_{i+1}$, where $\mathfrak g_{i+1}$ is defined by \eqref{degreek}.
In opposite direction, for fixed $\psi_1 \in \widehat P_i(t)$ and any $f\in \mathfrak g_{i+1}$ there exists   $\psi_2\in\widehat P_i(t)$ such that $\Pi_i^t(\psi_1)=\Pi_i^t(\psi_2)$ and$f=f_{\psi_1,\psi_2}^1$. By the complete analogy
with the proof of Lemma \ref{lem0} one gets the following identity
\begin{equation}
\label{structransi}
\bar c_{i+1}(\psi_2)=\bar c_{i+1}(\psi_1)+\bar\partial_{i+1} (f_{\psi_1,\psi_2}^{i+1}).
\end{equation}

Further take the space $\mathcal N_{i+1}$ as in \eqref{Nk}. By our assumptions on $\mathcal N$ we have

\begin{equation}
\label{normspliti}
{\rm Hom}(\mathfrak m, \mathfrak g_i/\mathfrak u_i)= \text{Im} \,\bar\partial_{i+1}\oplus\overline {\mathcal N}_{i+1}.
 \end{equation}

 Let $\widehat P_{i+1}$ be the subbundle of the bundle $\widehat P_i$ consisting of all $\psi\in \widehat P_i$ such that
 $\bar c_{i+1}(\psi)\in \overline{\mathcal N}_{i+1}$. From formula \eqref{structransi} and splitting \eqref{normspliti} it follows that the $\widehat P_{i+1}$ is not empty. Moreover, $\psi_1$ and $\psi_2$ from to the same fiber $\widehat P_{i+1}(t)$  satisfy $\Pi_i^t(\psi_1)=\Pi_i^t(\psi_2)$ if and only if $f_{\psi_1, \psi_2}^{i+1}\in \ker \bar\partial_i=\mathfrak u_{i+1}$.

 We also denote by $P_{i+1}$ the bundle over our curve with the fiber $P_{i+1}(t)$ over the point $\{\Lambda_j(t)\}_{j=-\mu}^{ -1}$ equal to the image of the  corresponding fiber $\widehat P_{i+1}(t)$ of the bundle $\widehat P_{i+1}$ under the map $\Pi_{i+1}^t$. By above $P_{i+1}$ can be considered as an affine bundle over $P_i$ such that its fibers are affine spaces over the vector space $\mathfrak u_{i+1}$. The projection of the bundle is induced by the maps $\Pi_i^t$. So, the bundles $\widehat P_{i+1}$ and $P_{i+1}$ satisfy conditions (A1)-(A3) above (for $k=i+1$).

To construct an Ehresmann connection $\widehat P_{i+1}$, satisfying condition (A4), consider
the
splitting
\begin{equation}
({\rm pr}_{i+1})^{-1}(\overline{\mathcal N}_{i+1})=\mathcal N_{i+1}\oplus{\rm Hom}(\mathfrak m, \mathfrak u_i).
\end{equation}
From this splitting and relation \eqref{freei} it follows that one can choose an Ehresmann connection  on $\widehat P_{i+1}$ such that
the degree i+1 component $c_{i+1}$ of the structure equation of $\widehat P_{i+1}$ with this Ehresmann connection belongs to $\mathcal N_{i+1}$. Moreover, any such Ehresmann connection on $\widehat P_{i+1}$ induces the unique Ehresmann connection on the bundle $P_i$.
This completes the induction step.

Since $\mathfrak u_k=0$ for $k\geq \mu$ and the maps $\Pi_i^t$ are identities for $i\geq \mu-1$, the sequence of bundles
$\{\widehat P_i\}_{i\geq 0}$ stabilizes from $i=\mu-1$, i.e. $\widehat P_i=\widehat P_{i+1}$ for $i\geq\mu-1$, and also $P_i=\widehat P_i$ for $i\geq\mu-1$.
Moreover, the bundle $P_\mu$ ($=\widehat P_\mu$) is endowed with the unique Ehresmann connection such that its structure function satisfies \eqref{Ehresk} with $k=\mu$.

We claim that the bundle $P_\mu$ with the constructed Ehresmann connection is the bundle we are looking for in Theorem \ref{mainthm}.
By construction the positive part of the structure  function  of $P_\mu$ takes values in $\mathcal N$.
Let us check that the bundle $P_\mu$ has symbol $\mathfrak u_+(\mathfrak m)$. By the previous constructions we have the sequence of the bundles $\{P_i\}_{i=0}^\mu$  such that $P_{i+1}$ is a bundle over $P_i$. Hence, $P_\mu$ can be considered as a bundle over each $P_i$ with $0\leq i<\mu$. Take $\psi\in P_\mu$ and consider the fibers of $P_\mu$ considered as a bundle of $P_i$ for each $0\leq i<\mu$. The tangent spaces to this fibers at $\psi$ define the filtration on $T_\psi P_\mu(t)$. Let $\omega$ and $L_\psi$ be as in \eqref{omega} and \eqref{Lpsi} respectively. The images of the spaces of the filtration on $T_\psi P_\mu(t)$ under the map $\omega(\psi)$
define the filtration on the space $L_\psi$ . By our constructions this filtration coincides with the filtration induced on $L_\psi$ from $\gl(V)$. Moreover, by property (A3) the graded space $\gr L_\psi$ of $L_\psi$ coincides with $\mathfrak u_+(\mathfrak m)$. In other words, the bundle $P_\mu$ has constant symbol  $\mathfrak u_+(\mathfrak m)$. So, $P_\mu$ satisfies all conditions of our theorem.
Moreover, all conditions we imposed on the structure function of $P_\mu$ during the proof  were necessary,
which proves the uniqueness of the constructed bundle and completes the proof of the first part of our theorem.

Note that by our constructions the bundle $P_\mu$ (without a specified Ehresmann connection on it) can be uniquely
determined by the subspace $\overline {\mathcal N}=\displaystyle{\bigoplus_{k>0}}\overline{\mathcal N_k}$ of $\Hom\bigl(\mathfrak m, \mathfrak g_+/\mathfrak u_+(\mathfrak m)\bigr)$: the bundle $P_\mu$ consists of all $\psi\in \widehat P_0$ such
that $\bar c_i(\psi)\in\overline{\mathcal N}_i$ for all $1\leq i\leq \mu$. The natural
action of $U_+(\mathfrak m)$ on $C^1_+$ defined by \eqref{action} induces the action on  $\Hom\bigl(\mathfrak m, \mathfrak g_+/\mathfrak u_+(\mathfrak m)\bigr)$.
Now assume that $\mathcal N$ is invariant with respect to the action of $U_+(\mathfrak m)$ on $C^1_+$.
Then the corresponding
space $\overline{\mathcal N}$ is invariant with respect to the action of $U_+(\mathfrak m)$ on $\Hom\bigl(\mathfrak m, \mathfrak g_+/\mathfrak u_+(\mathfrak m)\bigr)$. This implies that if
 $\psi\in P_\mu$ then
$\psi\circ B\in P_\mu$ for any $B\in U_+(\mathfrak m)$, which implies that the bundle $P_\mu$ is a principal $U_+(\mathfrak m)$-bundle. This completes the proof of the second part of our theorem.

\section{ Refinement procedure}
\label{refine}
\setcounter{equation}{0}
\setcounter{thm}{0}
\setcounter{lem}{0}
\setcounter{prop}{0}
\setcounter{cor}{0}
\setcounter{rem}{0}
\setcounter{exm}{0}

 Here we explain the origin of the condition of compatibility with respect to differentiation via the so-called refinement procedure in the cases of the standard representations  of the classical groups and the case $G={\rm Aff}(W)$.

\subsection{The case $G=SL(W)$, $GL(W)$, and ${\rm Aff}(W)$}
The constructions of this subsection work for any group $G$ such that any orbit $\mathcal O$ of flags
is compatible with respect to some grading of the corresponding Lie algebra $g$.
Starting with a curve of flags \eqref{filt} compatible with respect to differentiation, we can build a finer curve of flags compatible with respect to differentiation. For this first  given a curve $t\mapsto L(t)$  of subspaces of $W$ (parametrized somehow), i.e. a curve in a certain Grassmannian of $W$, denote by $C(L)$ the canonical bundle over $L$: the fiber of $C(L)$ over the point $L(t)$ is the vector space $L(t)$. Let $\Gamma(L)$ be the space of all sections of $C(L)$ . Set
 $L^{(0)}(t):=L(t)$ and  define inductively
$$L^{(j)}(t)={\rm span}\{\frac{d^k}{d\tau^k}\ell(t): \ell\in\Gamma(L), 0\leq k\leq j\}$$ for $j\geq0$.
The space $L^{(j)}(t)$ is called the \emph{$j$th extension} or the \emph{$j$th osculating subspace} of the  curve $L$ at the point $t$.
The compatibility of the curve of flags \eqref{filt} with respect to differentiation is equivalent to the condition $\Lambda_i^{(1)}(t)\subset \Lambda_{i-1}(t)$ for any $-\mu+1\leq i\leq -1$ and for any $t$.

Further given a subspace $L$ in $W$
denote by $L^\perp$ the annihilator of $L$ in the dual space $W^*$:
$L^\perp=\{p\in W^*:\langle p, v\rangle=0,\,\forall \,v\in L\}$. Set
$L^{(-j)}(t)=\bigl(\bigl(L(t)^\perp)^{(j)}\bigr)^\perp$
for $j\geq 0$. For any curve $t\mapsto L(t)$ in a Grassmannian we get the nondecreasing filtration
$\{L^{(j)}(t)\}_{j\in {\mathbb Z}}$ of $W$.
If dimensions $L^{(j)}(t)$ are independent of $t$ for any $j \geq 0$,
then the subspaces $L ^{(j)}(t)$, $j<0$, can be also defined
inductively by $L ^{(j)}(t)={\rm Ker}\frac{d}{dt}L ^{(j+1)}(t)$.


 Now take a germ of the curve \eqref{filt} at a point $t$ and assume that it is compatible with respect to differentiation.
 Assume that there exists a neighborhood $\mathcal U$ of $t$ such that one of the following assumptions hold:

 \begin{enumerate}
 \item[{\bf (B1)}] 
 For $\tau\in \mathcal U$ the dimension of subspaces $\bigl(\Lambda_i^{(1)}(\tau)\bigr)$ is constant, the subspace $\Lambda_i^{(1)}(\tau)$ is strictly between $\Lambda_i(t)$ and $\Lambda_{i-1}(t)$,  and
 one has the inclusion $\Lambda_i^{(2)}(\tau)\subset \Lambda_{i-1}(\tau)$;


 \item[{\bf (B2)}] For $\tau\in \mathcal U$ the dimension of the space   $\bigl(\Lambda_i^{(-1)}(\tau)\bigr)$ is constant, the subspace  $\Lambda_i^{(-1)}(\tau)$ is strictly between $\Lambda_{i+1}(t)$ and $\Lambda_i(t)$,  and
 one has the inclusion $\Lambda_{i+1}(\tau)\subset\Lambda_i^{(-2)}(\tau)$ (or, equivalently, $\Lambda_{i+1}^{(1)}(\tau)\subset\Lambda_i^{(-1)}(\tau)$)
      \end{enumerate}
Note that for $i=\mu+1$ the condition $\Lambda_i^{(2)}(\tau)\subset \Lambda_{i-1}(\tau)$ holds automatically in (B1) and for $i=-1$
the condition $\Lambda_{i+1}(\tau)\subset\Lambda_i^{(-2)}(\tau)$ holds automatically in (B2).

If assumption (B1) holds, then we can obtain a new germ of a curve of flags by inserting  in \eqref{filt} the space $\Lambda_i^{(1)}(\tau)$ between $\Lambda_i(\tau)$ and $\Lambda_{i-1}(\tau)$ for any $\tau\in \mathcal U$ . We call such operation \emph{elementary refinement of the first kind}. In the same way if assumption (B2) holds, then we can obtain a new germ of a curve of flag  by inserting  in \eqref{filt} the space $\Lambda_i^{(-1)}(\tau)$ between $\Lambda_{i+1}(\tau)$ and $\Lambda_{i}(\tau)$  for any $\tau\in \mathcal U$. We call such operation \emph{elementary refinement of the second kind}. In both cases we renumber the indices of subspaces in the obtained flags from $-1$ to $-\mu-1$.
%
The germ of a curve of flags $t$ is called a \emph{ refinement of the germ of the curve \eqref{filt} at $t$}, if it can be obtained from this germ
 by a sequence of elementary refinements (either of the first or the second kind).

By construction any refinement preserves the property of compatibility with respect to differentiation. Besides, if two curves are $GL(W)$-equivalent, then the corresponding refinements are $GL(W)$-equivalent.
Further, the refinement procedure defines a partial order on the set of all germs of smooth curves of flags in $W$, which also induces a partial order on the set of all refinements of the germ of the curve \eqref{filt} at $t$. In general a curve of flags admits several different maximal refinements with respect to this partial order.
Finally,  for a curve $t\mapsto L(t)$ in a Grassmannian of $W$ consider the corresponding trivial curve of flags $t\mapsto \{0\subset L(t)\subset W\}$. Then it is easy to show that in a generic point the germ of such curve of flags has the unique maximal refinement, which is exactly the curve $t\mapsto\{L^{(j)}(t)\}_{j\in {\mathbb Z}}$ of all osculating subspaces of the curve $L$ (truncated and renumbered in an obvious way). The corresponding map $\delta_t\in \displaystyle{\bigoplus_{j\in\mathbb Z}}{\rm Hom} \Bigl(L^j(t)/L^{j-1}(t),L^{j+1}(t)/L^{j}(t)\Bigr)$ satisfies the following properties: the restriction of $\delta_t$ on $L^j(t)/L^{j-1}(t)$  is surjective for $j<0$ and $\delta_t$ sends $(L^j(t)/L^{j-1}(t)$ \emph{onto} $L^{j+1}(t)/L^{j}(t)$.
  for $j\geq 0$. Hence the symbol of the curve $t\mapsto\{L^{(j)}(t)\}_{j\in {\mathbb Z}}$ satisfies similar properties.
  Note that the symbol of a curve of flags at a given point is in essence the tangent line to a curve.
  The symbol of the refined curve of flags contains an important information about the jet space of higher order of the original curve.  So, fixing the symbol of refined curve of flag instead of the original one, we fix more subtle classes of curves and our prolongation procedure  is more accurate:
the dimension of the algebraic prolongation  of the symbol of the refined curve (which in turn is equal the dimension of the canonical bundle of moving frames for the refined curve) might be significantly smaller than the dimension of the corresponding objects for the original curve.

\begin{exm}
\label{exm2}
(Non-degenerate curve in projective space)
{\rm Consider a curve  $t\mapsto L(t)$ in a projective space $\mathbb P W$ of an $n+1$-dimensional vector space $W$, i.e. a curve of lines in $W$. Without making a refinement procedure the symbol at a generic point of a curve, which does not degenerate to a point, is a line of degree $-1$ endomorphisms  of the corresponding graded space having rank 1 and the flat curve corresponds to a curve of lines in  $W$ depending linearly on a parameter ( or , equivalently, it is a projective line linearly embedded to $\mathbb P W$).
Obviously, the universal algebraic prolongation coincides with the subalgebra of $\gl(W)$ preserving the plane, generated by the lines of a flat curve.
Now consider so-called non-degenerate curves in $\mathbb P W$, i.e. curves which do not lie in any proper subspace of $\mathbb P W$. For such curve the $n$th osculating space at a generic point $t$ is equal to the whole $W$, $L^{(n)}(t)=W$. The maximal refinement of the curve $t\mapsto \{0\subset L(t)\subset W\}$ is the curve of complete flags
$$t\mapsto \{0\subset L(t)\subset L^{(1)}(t)\subset \ldots L^{(n)}(t)=W\},$$
as in Example \ref{projexm1} above. As already mentioned there
the symbol of such curve is a line of degree $-1$ endomorphisms of the corresponding graded spaces, generated by an endomorphism which has the matrix equal to a Jordan nilpotent block in some basis and
the curve in projective space corresponding to a flat curve of complete flags  is the curve of osculating subspaces of a rational normal curve in the projective space $\mathbb P W$. Finally, from 
Theorem \ref{prolongN} below (see also the sentence after it) it follows that the universal algebraic prolongation of this symbol is isomorphic to $\gl_2$ (or $\sll_2$ in the case $G=SL(W)$), which also follows from the well-known fact that the algebra of infinitesimal symmetries of a rational normal curve is isomorphic to $\gl_2$ (or $\sll_2$ in the case $G=SL(W)$). In this case one can show that the normalization condition can be chosen to be  invariant with respect to the natural action of the subgroup $U_+(\mathfrak m)$ (which is isomorphic in this case to the group the upper-triangular matrices) on $C^1_+$ and in this way the classical complete system of invariants of curves in projective spaces, the Wilczynski invariants, can be constructed (see \cite{dzpar} or \cite{doub3} for the detail). $\Box$
}
\end{exm}

 \subsection{The case $G=Sp(W)$ or $G=CSp(W)$}
Here $W$ is equipped with a symplectic form $\sigma$.
In this case, by Proposition \ref{compsymp} our theory works for curves of \emph{symplectic flags} as in Definition \ref{sympflagdef}.
Given a curve $t\mapsto L(t)$ in a Grassmannian of $W$ denote by  $t\mapsto L^\angle(t)$ the curve of subspace of $W$ such that the space $L^\angle(t)$ is the skew-symmetric complement of
the space $L(t)$ with respect to the symplectic form $\sigma$.
Starting with a germ of a curve \eqref{filt} of symplectic flags  compatible with respect to differentiation and applying an elementary refinement as in the previous subsection, we obtain non-symplectic flags in general. Therefore we need to modify the definition of an elementary refinement appropriately. Such modification is based on the following simple fact: if $t\mapsto L(t)$ is a curve of coisotropic or isotropic subspaces of $W$
then $(L^{(1)})^\angle(t)=(L^{\angle})^{(-1)}(t)$.
We say that a refinement of a curve of symplectic flag $t\mapsto \{\Lambda_i(t)\}_{i\in \mathbb Z}$ is an elementary symplectic refinement in one of the following 2 cases:

\begin{enumerate}
\item[{\bf (C1)}]
If condition (B1) holds with subspaces $\Lambda_i(\tau)$ being coisotropic, then an \emph{elementary symplectic refinement of the  first kind} consists of inserting $\Lambda_i^{(1)}(\tau)$ between $\Lambda_i(\tau)$ and $\Lambda_{i-1}(\tau)$ and inserting $(\Lambda_i^\angle)^{(-1)}(\tau)$ between $(\Lambda_{i-1})^\angle(\tau)$ and $(\Lambda_{i})^\angle(\tau)$ for any $\tau\in \mathcal U$;

\item[{\bf (C2)}]
If condition (B2) holds with subspaces $\Lambda_{i+1}(\tau)$ being coisotropic, then an \emph{elementary symplectic refinement of the second kind} consists of inserting $\Lambda_i^{(-1)}(\tau)$ between $\Lambda_{i+1}(\tau)$ and $\Lambda_{i}(\tau)$ and inserting $(\Lambda_i^\angle)^{(1)}(\tau)$ between $(\Lambda_{i})^\angle(\tau)$ and $(\Lambda_{i+1})^\angle(\tau)$ for any $\tau\in \mathcal U$;
\item[{\bf (C3)}]
If $\Lambda_i(\tau)$ is isotropic and $\Lambda_{i-1}(\tau)$ is coisotropic (which implies by assumptions that $\Lambda_{i-1}(\tau)=\Lambda_i^\angle(\tau)$, then assume that $(\Lambda_i)^{(1)}(\tau)$ is isotropic, of constant dimension in $\mathcal U$, $\Lambda_i(\tau)$ is strictly contained in  $(\Lambda_i)^{(1)}(\tau)$,  and  $\Lambda_i^{(2)}\subset \Lambda_{i-1}^{(-1)}(\tau)$. An \emph{elementary symplectic refinement of  the third kind} consists of inserting $\Lambda_i^{1}(\tau)$ and $\Lambda_{i-1}^{(-1)}(\tau)$ between $\Lambda_i(\tau)$ and $\Lambda_{i-1}(\tau)$
(if $\Lambda_i^{1}(\tau)$ and $\Lambda_{i-1}^{(-1)}(\tau)$ coincide they count as a one space).
\end{enumerate}

By analogy with the previous subsection one can define a symplectic refinement as a composition of elementary symplectic refinement. By construction the resulting flag under a symplectic refinement is symplectic and compatible with respect to diferentiation.
Also the refinement procedure defines a partial order on the set of all germs of smooth curves of symplectic flags in $W$. In general, similar to the previous subsection,  a curve of symplectic flags admits several different maximal symplectic refinements with respect to this partial order. There is a unique maximal refinement for the germs at generic points of the following two types of curves of symplectic flags:
$t\mapsto \{0\subset L(t)\subset W\}$, where spaces $L(t)$ are Lagrangian, i.e. isotropic of dimension $\frac{1}{2}\dim W$, and
$t\mapsto \{0\subset L(t)\subset L^\angle(t)\subset W\}$, where spaces $L(t)$ are proper isotropic.
In the first cases we can use elementary refinement of the first kind only and the maximal refinement coincides with the flag associated with a curve in a Lagrangian Grassmannian, introduced in \cite{zelli1},\cite{zelli2}. In the second case one can use elementary refinements of the first and third kinds only.

 \subsection{The case $G=O(W)$ or $G=CO(W)$}
Here $W$ is equipped with a non-degenerate symmetric form $Q$.
The elementary refinement are defined in completely the same way as in the symplectic case: isotropic and coisotropic subspaces are taken with respect to the form $Q$ and instead of skew-symmetric complement one takes orthogonal complements with respect to $Q$.

 \section{Classification of symbols of curves of flags with respect to
classical groups}
\setcounter{equation}{0}
\setcounter{thm}{0}
\setcounter{lem}{0}
\setcounter{prop}{0}
\setcounter{cor}{0}
\setcounter{rem}{0}
\setcounter{exm}{0}
\label{symbclasec}

We describe all symbols of curves of flags with respect to $GL(W)$ (equivalently, $SL(W)$) , $Sp(W)$ (equivalently,  $CSp(W)$), and
$O(W)$ (equivalently, $CO(W)$.
According to Remark \ref{symbrepind} it can be used for any representation of classical groups.

\subsection{The case of $GL(W)$($SL(W)$)}
\label{symglsec}
Let $\delta_1$ and $\delta_2$ be degree $-1$ endomorphisms of the graded spaces $V_1$ and $V_2$, respectively. The direct sum $V_1\oplus V_2$ is equipped with  the natural grading such that its $i$th component is the direct sum of $i$th components of $V_1$ and $V_2$.
The direct sum $\delta_1\oplus\delta_2$ is the degree $-1$ endomorphism of $V_1\oplus V_2$ such that the restriction of it to $V_i$ is equal to $\delta_i$ for each $i=1,2$.  A degree $-1$ endomorphism $\delta$ of a graded space $V$ is called \emph{indecomposable} if
it cannot be represented as a direct sum of two degree $-1$ endomorphisms acting on nonzero graded spaces.
Further, given two integers $r\leq s<0$ let $V_{rs}=\displaystyle{\bigoplus_{i=r}^sE_i}$, where $\dim E_i=1$ for every $i$, $r\leq i\leq s$, and let $\delta_{rs}$ be a degree $-1$ endomorphism of $V_{rs}$ which sends $E_i$ onto $E_{i-1}$ for every $i$, $r<i\leq s$, and sends $E_r$ to $0$. Then clearly $\delta_{rs}$ is indecomposable.

The following theorem gives the classification of all symbols  of curves of flags (both of unparametrized and parametrized)  with respect to the General Linear group:

\begin{thm}
\label{glclass}
A degree $-1$ endomorphism $\delta$ of a graded space (indexed by negative integers) is conjugated to the direct sum of endomorphisms of type $\delta_{rs}$.
Moreover, for any two integers $r\leq s<0$ the number of apperance of $\delta_{rs}$ in this direct sum
is
an invariant
of the conjugate class.
\end{thm}

\begin{proof}
We will proceed as in the classical proof of Jordan Normal Form Theorem (see, for example, \cite[chapter 3]{gelf}). The only difference is that one has to take into account the grading on the ambient space $V$ as well.
Let $\delta$ be a degree $-1$ endomorphism of $V$. Then $\delta$ defines the additional filtration on $V$ via generalized eigenspaces of different orders (corresponding to the unique eigenvalue $0$), namely $\ker (\delta)\subset ker (\delta^2)\subset\ldots \ker (\delta^{l+1})=V$. In the first step choose a tuple of vectors in $V$ consisting of homogeneous vectors (with respect to the grading on $V$) such that their images under the canonical projection to the factor space $\ker (\delta^ {l+1})/\ker (\delta^{l})$ constitute a basis of
$\ker (\delta^ {l+1})/\ker (\delta^{l})$. Then take the images of vectors of  this tuple under $\delta$. Since $\delta$ is of degree $-1$, these images are homogeneous vectors in $V$ belonging to $\ker(\delta^{l})$. Complete the tuple of these images (if necessary) to a tuple of homogeneous vectors in $\ker (\delta^{l})$ such that the images of vectors in this tuple under the canonical projection to the factor space $\ker (\delta^ {l})/\ker (\delta^{l-1})$ constitute a basis of $\ker (\delta^ {l})/\ker (\delta^{l-1})$. Continuing this process we will get a basis of $V$ consisting of homogeneous vectors such that the matrix of the operator $\delta$ in this bases has a Jordan normal form. Each Jordan block corresponds to an indecomposable endomorphism of type $\delta_{rs}$. This proves that $\delta$ is conjugated to the direct sum of endomorphisms $\delta_{rs}$.
The invariance of the number of appearance of $\delta_{rs}$ in this direct sum
can be easily obtained from the above construction as well.
\end{proof}
The previous theorem also shows that \emph{the endomorphisms $\delta_{rs}$ are the only indecomposable degree $-1$ endomorphisms, up to a conjugation, of a graded space (indexed by negative integers)}.

\subsection{Symplectic case}
\label{symbspsec}
Let $\frac{1}{2}\mathbb Z:=\{\frac{i}{2}: i\in \mathbb Z\}$ and $\frac{1}{2}\mathbb Z_{\rm{odd}}:=\mathbb Z\backslash \mathbb Z_{{\rm odd}}=\{i+\frac{1}{2}: i\in \mathbb Z\}$.
It is more convenient for a symplectic flag to make a shift in the indices (either by an integer or by a half of an integer) such that $(\Lambda_{i})^\angle=\Lambda_{-i+frac{1}{2}
}$ for any $i$ in $\mathbb Z$ or $\frac{1}{2}\mathbb Z_{\rm{odd}}$.
In the sequel the subspaces in symplectic flags will be enumerated according to this rule.

Assume that $V=\displaystyle{\bigoplus_{i=r}^sV_i}$ is a graded space, $V_r\neq 0$, $V_s\neq 0$,  and a symplectic form $\sigma$ is given $V$.
We say that this grading is \emph{symplectic} (or $V$ is a \emph{graded symplectic space}) if
the corresponding flag
\begin{equation}
\label{gradsympflag}
\{V^i\}_{i\in \mathbb Z \text{ or } \frac{1}{2}\mathbb Z_{\rm{odd}}},\text{  where } V^i=\displaystyle{\bigoplus_{j\geq i}V_j},
\end{equation}
 is symplectic and also the
subspaces $V_i$ and $V_j$ are skew orthogonal for all pairs $(i,j)$ with $i+j<r+s$.
As in the previous subsection, the notion of indecomposibility plays the crucial role in the classification.
Let $\delta_1$ and $\delta_2$ be degree $-1$ endomorphisms of the graded symplectic spaces $V_1$ and $V_2$, respectively, belonging to the corresponding symplectic algebras. The direct sum $V_1\oplus V_2$ is equipped with  the natural symplectic grading such that its $i$th component is the direct sum of $i$th components of $V_1$ and $V_2$ and the symplectic form on  $V_1\oplus V_2$ is defined as follows: the restriction of this form to $V_i$ coincides with the symplectic form on $V_i$ for each $i=1,2$ and the spaces $V_1$ and $V_2$ are skew-orthogonal in $V_1\oplus V_2$.
The direct sum $\delta_1\oplus\delta_2$ is the degree $-1$ endomorphism of $V_1\oplus V_2$ (belonging to $\mathfrak{sp}(V_1\oplus V_2)$) such that the restriction of it to $V_i$ is equal to $\delta_i$ for each $i=1,2$.
 A degree $-1$ endomorphism $\delta\in \mathfrak{sp}(V)$ of a graded symplectic space $V$ is called \emph{symplectically indecomposable} if
 it cannot be represented as a direct sum  of two degree $-1$ endomorphisms acting on nonzero graded spaces and belonging to the corresponding symplectic algebras.

\begin{rem}
\label{halfintrem}
{\rm
Note that we can add two graded symplectic spaces and two degree $-1$ endomorphisms on them if and only if either the grading on both spaces are indexed by $\mathbb Z$ or the grading on both spaces are indexed by $\frac{1}{2} \mathbb Z_{\rm{odd}}$.}$\Box$
\end{rem}
Below we list two types of symplectically indecomposable degree $-1$ endomorphism:
\begin{enumerate}
\item[\bf{(D1)}]
Given a nonnegative $s\in \frac{1}{2}\mathbb Z$, and an integer $l$   such that $0\leq l\leq 2[s]$ let $V_{s;l}^{\mathfrak{sp}}$ be a linear symplectic space with a basis
\begin{equation}
\label{tuple1}
\{e_{s-l},\ldots, e_s,f_{-s},\ldots,f_{l-s}\}
\end{equation}
such that $\sigma(e_i,e_j)=\sigma(f_i,f_j)=0$, $\sigma(e_i,f_j)=(-1)^{s-i}$, if $j=-i$, and $\sigma(e_i,f_j)=0$ if  $j\neq -i$. Define the grading on $V_{s;l}^{\mathfrak{sp}}$ such that the $i$th component equal to the span of all vectors with index $i$ appearing in the tuple \eqref{tuple1}. It is a symplectic grading. Then denote by $\delta_
{s;l}^{\mathfrak{sp}}$ a degree $-1$ endomorphism of
$V_
{s;l}^{\mathfrak{sp}}$ from the symplectic algebra such that $\delta_{s;l}^{\mathfrak{sp}}(e_i)=e_{i-1}$ for $s-l+1\leq i\leq s$, $\delta_{s;l}^{\mathfrak{sp}}(e_{s-l})=0$, $\delta_{s;l}^{\mathfrak{sp}}(f_i)=f_{i-1}$ for $-s+1\leq i\leq l-s$, and
$\delta_{s;l}^{\mathfrak{sp}}(f_{-s})=0$.


\item[\bf{(D2)}]
Given a positive $m\in \frac{1}{2} Z_{\rm {odd}}$ let
$$\mathcal L_m^{\mathfrak{sp}}=\displaystyle{\bigoplus_{\tiny{\begin{array}{c}-m\leq i\leq m,\\i\in\frac{1}{2}\mathbb Z_{\rm {odd}}\end{array}}}}E_i$$
be a symplectic graded spaces such that $\dim E_i=1$ for every admissible $i$ and let $\tau_m^{\mathfrak{sp}}$ be a degree $-1$ endomorphism of $\mathcal L_m$ from the symplectic algebra which sends $E_i$ onto $E_{i-1}$ for every admissible $i$, except  $i=-m$, and  $\tau_m^{\mathfrak{sp}}(E_{-m})=0$. In the case $\mathbb K=\mathbb R$ we also assume that $\sigma(\tau_m^{\mathfrak{sp}} (e), e)\geq 0$ for all $e_0\in E_{\frac{1}{2}}$.
\end{enumerate}
By Remark \ref{halfintrem} we cannot the direct sum of $\delta_{s;}^{\mathfrak{sp}}$ with $s\in \frac{1}{2}\mathbb Z$ and $\tau_m^{\mathfrak{sp}}$, for example.
The following theorem gives the classification of all degree $-1$ endomorphisms from the symplectic algebra of a graded symplectic space $V$ and consequently the classification of all symbols  of curves (both of unparametrized and parametrized) of symplectic flags of
$V$ with respect to $Sp(V)$ and $CSp(V)$:

\begin{thm}
\label{spclass1}
Assume that $V$ is a graded symplectic space.

\begin{enumerate}
\item
If the grading on $V$ is indexed by $\mathbb Z$ then
a degree $-1$ endomorphism from $\mathfrak{sp}(V)$
is conjugated (by a symplectic transformation)  to the direct sum of endomorphisms of type $\delta_{s;l} ^{\mathfrak{sp}}$, where $s$ is a nonnegative integer
and $0\leq l \leq 2s$. Moreover, for each pair of integers $(s,l)$ with $0\leq l\leq 2s$ the number of appearances of $\delta_{s;l} ^{\mathfrak{sp}}$ in this direct sum is an invariant of the conjugate class.
\item
If the grading on $V$ is indexed  by $\frac{1}{2}\mathbb Z_{\rm {odd}}$ and
$\mathbb K=\mathbb C$, then
a degree $-1$ endomorphism from $\mathfrak{sp}(V)$
 is conjugated (by a symplectic transformation) to the direct sum of endomorphisms of type $\delta_{s;l}^{\mathfrak{sp}}$ with $s\in \frac{1}{2}\mathbb Z_{\rm {odd}}$ and  $0\leq l\leq 2s-1$, and of type  $\tau_m^{\mathfrak{sp}}$ with positive $m\in \frac{1}{2}\mathbb Z_{\rm {odd}}$. Moreover, for each pair $(s,l)\in \frac{1}{2}\mathbb Z_{\rm {odd}}\times\mathbb Z$ with $0\leq l\leq 2s-1$  and a positive $m\in \frac{1}{2}\mathbb  Z_{\rm {odd}}$ the numbers of appearances of $\delta_{s;l} ^{\mathfrak{sp}}$ and $\tau_m ^{\mathfrak{sp}}$ in this direct sum are  invariants of the conjugate class.

\item
If the grading on $V$ is indexed by  $\frac{1}{2}\mathbb Z_{\rm {odd}}$ and
$\mathbb K=\mathbb R$, then
a degree $-1$ endomorphism from $\mathfrak{sp}(V)$
is conjugated  to the direct sum of endomorphisms of type $\delta_{s;l} ^{\mathfrak{sp}}$, $\tau_m ^{\mathfrak{sp}}$, and
$-\tau_m ^{\mathfrak{sp}}$.
Moreover, for each pair $(s,l)\in \frac{1}{2}\mathbb Z_{\rm {odd}}\times\mathbb Z$ with $0\leq l\leq 2s-1$  and a positive $m\in \frac{1}{2} \mathbb Z_{\rm {odd}}$ the numbers of appearances of $\delta_{s;l} ^{\mathfrak{sp}}$, $\tau_m ^{\mathfrak{sp}}$, and
$-\tau_m ^{\mathfrak{sp}}$  in this direct sum are  invariants of the conjugate class.
\end{enumerate}
\end{thm}

Theorem \ref{spclass1} shows that, up to a conjugation, the only symplectically indecomposable endomorphisms are the endomorphisms of the
type $\delta_{s;l}^{\mathfrak{sp}}$ and $\tau_m^{\mathfrak{sp}}$ for $\mathbb K=\mathbb C$ and of the type $\delta_{s;l}^{\mathfrak{sp}}$, $\tau_m^{\mathfrak{sp}}$, and $-\tau_m^{\mathfrak{sp}}$ in the case $\mathbb K=\mathbb R$.

\begin{proof}
Let $\delta$ be a degree $-1$ endomorphism of $V$ belonging to $\mathfrak{sp}(V)$.  Assume that $\delta^{l+1}=0$ and $\delta^{l}\neq 0$ for some $l\geq 0$. Let $\mathcal C$ be a complement to $\ker \delta^{l}$ in $V$, i.e.
\begin{equation}
\label{sympsymb1}
V=\mathcal C\oplus \ker \delta^{l}.
\end{equation}
Define a bilinear form $b:\mathcal C \times \mathcal C\mapsto\mathbb K$ by
\begin{equation}
\label{formb}
b(u_1,u_2):=\sigma(u_1,\delta^{l} u_2).
\end{equation}
Since $\delta\in
\mathfrak{sp}(V)$ and $\sigma$ is skew-symmetric, the form $b$ is symmetric if $l$ is odd and skew-symmetric if $l$ is even.
As a matter of fact the form $b$ can be considered as a bilinear form on $V/\ker \delta^{l}$.
\begin{lem}
\label{lemb}
The form $b$ is non-degenerate.
\end{lem}
\begin{proof}
Assume that there exists $u_1\in \mathcal C$ such that $b(u_1,u_2)=0$ for any $u_2\in \mathcal C$. In other words, $\sigma(u_1,\delta^{l} u_2)=0$ for any $u_2\in \mathcal C$. This together with splitting \eqref{sympsymb1} implies that $\sigma(u_1,\delta^{l} v)=0$ for any $v\in V$.
Note that from the fact that $\delta\in
\mathfrak{sp}(V)$ it follows that $\sigma(u_1,\delta^{l}v)=(-1)^{l}\sigma(\delta^{l}u_1,v)$, therefore
$\sigma(\delta^{l}u_1,v)=0$ for any $v\in V$. Since $\sigma$ is non-degenerate, we get $\delta^{l}u_1=0$, i.e. $u_1\in \ker \delta^{l}$.  As a consequence of this, our assumption that $u_1\in \mathcal C$, and splitting \eqref{sympsymb1} we get that $u_1=0$. This completes the proof of the lemma.
\end{proof}

Further from the splitting \eqref{sympsymb1} it follows that $\delta^i \mathcal C\cap\delta^j \mathcal C=0$ for $i>j$. Set

\begin{equation}
\label{sympsymb2}
Y_{l}:=\mathcal C\oplus\delta \,\mathcal C\oplus\ldots\oplus \delta^{l} \mathcal C.
\end{equation}

\begin{lem}
\label{lemnon}
The restriction of the symplectic form $\sigma$ to the subspace $Y_{l}$ is non-degenerate.
\end{lem}
\begin{proof}
From Lemma \ref{lemb} it follows that for any $i$ the bilinear form $(u_1,u_2)\mapsto \sigma(\delta^i u_1,\delta^{l-i} u_2)$ is non-degenerate. Also the condition $\delta^{l+1}=0$ implies that $\sigma(\delta^i u_1,\delta^j u_2)=0$ for $i+j\geq l$. This implies that the matrix of the form $\sigma|_{Y_{l}}$ with respect to any basis of $Y_{l}$ is block-triangle with respect to non-principal diagonal and each non-principal diagonal block is nonsingular. This completes the proof of the lemma.
 \end{proof}

We can always choose the space $\mathcal C$ in \eqref{sympsymb1} such that $\mathcal C$ is a direct sum of homogeneous spaces (i.e. a direct sum of subspaces of $V_i$).
By the previous lemma, $V=Y_{l}\oplus Y_l^\angle$, where $Y_{l}^\angle$ is the skew-symmetric complement of $Y_{l}$.
By constructions $ Y_l^\angle\subset\ker\delta^{l}$. Also it is easy to show that $Y_l^\angle$ inherits the symplectic grading from $V$,
$Y_l^\angle=\displaystyle{\bigoplus_{i\in\mathbb Z}Y_l^\angle\cap V_i}$.
Repeat the same procedure for $Y_{l}^\angle$ instead  of $V$, then, if necessary, repeat it again. In this way one gets the unique skew-orthogonal splitting of $V$ into the direct sum of invariant subspaces
of $\delta$,
\begin{equation}
\label{sympsymb3}
V=\bigoplus _{i=1}^dY_{l_i}, \quad l=l_1>l_2>\ldots >l_d,
\end{equation}
such that $\delta^{l_i+1}|_{Y_{l_i}}=0$, $\delta^{l_i}|_{Y_{l_i}}\neq 0$, $Y_{l_i}=\mathcal C_i\oplus\delta\, \mathcal C_i\oplus\ldots\oplus \delta^{l_i} \mathcal C_i$ for a complement $\mathcal C_i$ to $\ker \delta^{l_i}|_{Y_{l_i}}$ in $Y_{l_i}$, and each $Y_{l_i}$ is a graded symplectic space with the grading inherited from $V$.

From the splitting \eqref{sympsymb3} it follows that to prove our theorem it is sufficient to restrict ourself to the case when $V=Y_l$, where $Y_l$ is as in \eqref{sympsymb2}.
Assume that $s$ is the maximal nontrivial degree in the grading of $V$
(i.e $V_s\neq 0$ and $V_i=0$ for $i>s$).
Let $Z_s$ be the space of all vectors of degree $s$ in $V$. Then by constructions $Z_s$ is transversal to $\ker \delta^{l}$.
Let
$$\mathcal X=Z_s\oplus\delta Z_s\oplus\ldots\oplus\delta^{l}Z_s$$
and $\widetilde Z_s$ be the image of $Z_s$ in $V/\ker \delta^{l}$ under the canonical projection on the quotient space. In the sequel we will look on the form $b$ as on a bilinear form on $V/\ker \delta^{l}$.
\begin{lem}
The space $\mathcal X\cap\mathcal X^\angle$ is equal to $0$ or $\mathcal X$.
\end{lem}
\begin{proof}
By the same arguments as in the proof of Lemma \ref{lemnon} the statement of the present lemma is equivalent to the following statement: the restriction $b|_{\widetilde Z_s}$ of the form $b$ on the space $\widetilde Z_s$ is either equal to $0$ identically or non-degenerate.
Set $K_1=\ker b|_{\widetilde Z_s}$. Assume by contradiction that $K_1$ is a nonzero proper subspace of $Z_s$.
Let $\{V^j\}_{j\in\mathbb Z \text{ or } \frac{1}{2}\mathbb Z_{\tiny{\rm odd}}}$ be the flag as in \eqref{gradsympflag}.
From the assumption that $K_1\neq \widetilde Z_s$ it follows that $\delta Z_s\nsubseteq (V^{s-l+2})^\angle$. Hence
\begin{equation}
\label{Zs1}
V^{s-1}\nsubseteq (V^{s-l+2})^\angle.
\end{equation}
Since the flag $\{V^j\}_{j\in\mathbb Z}$ is symplectic, the space $(V^{s-l+2})^\angle$ has to be equal to one of the subspaces of the flag $\{V^j\}_{j\in\mathbb Z}$. This together with the assumption that $s$ is the maximal degree in the grading of $V$ and relation \eqref{Zs1} implies that
\begin{equation}
\label{Zs2}
(V^{s-l+2})^\angle=V_s=Z_s.
\end{equation}
Now  let $K_1^\perp=\{v\in V/\ker \delta^{l}: b(v,u)=0 \quad \forall u\in K_1
\}$. By constructions, $\widetilde Z_s\subseteq K_1^\perp$. The assumption $K_1\neq 0$ implies that $K_1^\perp$ is a proper subspace of $V/\ker \delta^{l}$.
 Therefore there exist elements of degree less than $s$  in $V$ that are transversal to $\ker \delta^{l}$. Consequently the space $V^{s-l+1}$ is a proper subspace of $V$ and $V^{s-l+2}\subsetneqq V^{s-l+1}$
This yields in turn that that $(V^{s-l+1})^\angle$ is a nonzero proper subspace of $(V^{s-l+2})^\angle=V_s$, which contradicts
the fact that the flag $\{V^j\}_{j\in\mathbb Z}$ is symplectic and that $s$ is the maximal degree in the grading of $V$.
The proof of the lemma is completed.
\end{proof}

Now consider separately the cases $\mathcal X\cap\mathcal X^\angle=0$ and $\mathcal X\cap\mathcal X^\angle=\mathcal X$.

\begin{lem} If $\mathcal X\cap\mathcal X^\angle=0$, then $V=\mathcal X$.
\end{lem}

\begin{proof}
$V=\mathcal X\oplus\mathcal X^\angle$. Let us prove that $\mathcal X^\angle=0$ or, equivalently, $V=\mathcal X$. Indeed,
assume by contradiction that $\mathcal X^\angle\neq 0$. Then $\mathcal X^\angle$ is a graded symplectic spaces with the symplectic forms
and the gradings inherited from $V$ such that the maximal degree in the grading is less than $s$. This together with the assumption that $V=Y_l$, where $Y_l$ is as in \eqref{sympsymb2}, implies that in the grading of $V$ there are nonzero vectors of degree less than $s-l+1$. This contradicts relation
\eqref{Zs2} that must hold in the considered case.
\end{proof}

\begin{lem}
\label{spclasscor}
Assume that $\mathcal X\cap\mathcal X^\angle=0$. Then the following three statements hold:
\begin{enumerate}
\item
If $l$ is even, then the endomorphism $\delta$ is conjugated to the direct sum of the endomorphisms $\delta_{s,2s}^{\mathfrak{sp}}$ with an integer $s$ repeated $\frac{1}{2}\dim Z_s$ times;
\item
If $l$ is odd and  $\mathbb K=\mathbb C$, then the endomorphism $\delta$ is conjugated to the direct sum of the endomorphisms $\tau_{s}^{\mathfrak{sp}}$ repeated
$\dim Z_s$ times;

\item
If $l$ is odd and $\mathbb K=\mathbb R$, then the endomorphism $\delta$ is conjugated to the direct sum of the endomorphisms $\tau_{s}^{\mathfrak{sp}}$ and $-\tau_{s}^{\mathfrak{sp}}$
such that the numbers of appearances of  $\tau_{s/2}$ and $-\tau_{s/2}$ are equal to the positive and the negative indices of the form $b$, respectively.
\end{enumerate}
\end{lem}

\begin{proof}
In the considered case as a subspace $\mathcal C$ in \eqref{sympsymb1} one can take $Z_s$.
As was already mentioned before, the form $b$ (defined on $Z_s$) is non-degenerate symmetric if $l$ is odd and non-degenerate skew-symmetric if $l$ is even.
If $l$ is even we can choose a Darboux (symplectic) bases in $Z_s$ with respect to the form $b$, i.e. a basis
$\{\varepsilon_j,\nu_j\}_{j=1}^{\frac{1}{2}\dim Z_s}$ such that the form $b$ is nonzero (and equal to $\pm 1$)
only for the pairs $(\varepsilon_j,\nu_j)$ and $(\nu_j,\varepsilon_j)$. Then for each $1\leq j\leq \frac{1}{2}\dim Z_s$ the restriction of the endomorphism $\delta$ to the minimal invariant subspaces of $\delta$ containing $\varepsilon_j$ and $\nu_j$ is conjugated to $\delta_{s;2s}^{\mathfrak{sp}}$, which proves item (1) of the lemma. If $l$ is odd we choose a basis of $Z_s$ for which the quadratic form corresponding to the form $b$ is diagonal. Then for each vector of this basis the restriction of $\delta$ to the minimal invariant subspaces of $\delta$ containing this vector is conjugated to $\tau_s^{\mathfrak{sp}}$, which proves items (2) and (3) of the lemma.
\end{proof}

Now consider the case  $\mathcal X\cap\mathcal X^\angle=\mathcal X$.
In this case  $\widetilde Z_s\subset \widetilde Z_s^\perp$. The grading on $V$ induces the natural grading on $V/\ker \delta^{l}$.
Take a subspace $\widetilde K_2$ consisting of the homogeneous elements of the minimal degree $s_1$ in $V/\ker \delta^{l}$.
By constructions $s_1<s$. Let us prove that
\begin{equation}
\label{Zs3}
V/\ker \delta^{l}=\widetilde Z_s^\perp\oplus \widetilde K_2.
\end{equation}
Take a subspace $K_2$ which is a representative of $\widetilde K_2$ in $V$ and consists of homogeneous vectors (of degree $s_1$).
From minimality of $s_1$ it follows that $V_{s_1-l+1}=\delta^{l}K_2$ and $V_i=0$ for $i>s_1-l+1$. From the maximality of $s$ it follows that $Z_s=(V^{s_1-l+2})^\angle$.  This implies that $\dim\widetilde K_2=\dim \widetilde Z_s$ and that $\widetilde K_2\cap \widetilde Z_s^\perp= 0$, which in turn yields \eqref{Zs3}.

Now let
$\mathcal Y=K_2\oplus\delta K_2\oplus\ldots\oplus\delta^{l}K_2$.
By constructions, the restriction of the endomorphism $\delta$ to
$\mathcal X \oplus \mathcal Y$ is conjugated to the direct sum of endomorphisms $\delta_{s;l}^{\mathfrak{sp}}$ repeated $\dim Z_s$ times
Further consider the space $(\mathcal X \oplus \mathcal Y)^\angle$.
This space inherits the grading from $V$,
$(\mathcal X \oplus \mathcal Y)^\angle=\displaystyle{\bigoplus_{i\in\mathbb Z}(\mathcal X \oplus \mathcal Y)^\angle\cap V_i}$, with the maximal nontrivial grading less than $s$.
Repeat the same procedure for $(\mathcal X \oplus \mathcal Y)^\angle$ instead  of $V$, then, if necessary, repeat it again. In this way we decompose the endomorphism $\delta$ in the case of $V=Y_l$ to the direct sum of symbols of the type $\delta_{s;l}^{\mathfrak{sp}}$ and $\tau_m^{\mathfrak{sp}}$ for $\mathbb K=\mathbb C$ or $\delta_{s;l}^{\mathfrak{sp}}$, $\tau_m^{\mathfrak{sp}}$, and $-\tau_m^{\mathfrak{sp}}$ in the case $\mathbb K=\mathbb R$. This together with the decomposition \eqref{sympsymb3}  completes the proof of Theorem \eqref{spclass1}.
\end{proof}

\begin{rem}
\label{condGrem}
\rm{In the works \cite{zelli1,zelli2} of the second author with C. Li parameterized curves in Lagrangian Grassmannian (over $\mathbb R$) satisfying so-called condition (G) were considered. In the present terminology condition (G) of  \cite{zelli1,zelli2} is equivalent to the condition that the symbol of the parameterized curve of symplectic flag is conjugated to the direct sum of endomorphisms of type $\tau_m$ and $-\tau_m$.
The Young diagram which were assigned there to a curve in Lagrangian Grassmannian can be described as follows: for any $p\in  \mathbb N$
the number of rows of length $p$ in it is equal to the number of appearances of endomorphisms $\tau_{\tiny{\frac{2p-1}{2}}}^{\mathfrak{sp}}$ and $-\tau_{\tiny{\frac{2p-1}{2}}}^{\mathfrak{sp}}$ in this direct sum. We calculate the universal prolongation of such symbols in subsection \ref{condGsubsec} below, which together with Theorem \ref{mainthm} gives more conceptual point of view on the constructions of papers \cite{zelli1,zelli2} and generalize them to more general classes of curves.}
\end{rem}

\subsection{Orthogonal case}
\label{orthsymbsec}
The classification of symbols in this case  is very similar to the symplectic case.
As in the symplectic case we make a shift in the indices (either by an integer or by a half of an integer) such that $(\Lambda_{i})^\perp=\Lambda_{-i+1}$ for any $i$ in $\mathbb Z$ or $\frac{1}{2}\mathbb Z_{\rm{odd}}$, where $L^\perp$ denotes the orthogonal complement of a subspace $L$ with respect to the nondegenerate symmetric form $Q$.
Further, by complete analogy with the symplectic space we can define graded orthogonal spaces and orthogonally indecomposable degree $-1$ endomorphisms.

By analogy with (D1) and (D2) of the previous subsection, there are the following two types of orthogonally indecomposable degree $-1$ endomomorphisms:
\begin{enumerate}
\item[\bf{(E1)}]
Given a positive $s\in \frac{1}{2}\mathbb Z$, and an integer $l$   such that $0\leq l\leq 2(s+\{s\})-1$ let $V_{s;l}^{\mathfrak{so}}$ be a a linear space equipped with a nondegenerate symmetric form and with a basis
\begin{equation}
\label{tuple1}
\{e_{s-l},\ldots, e_s,f_{-s},\ldots,f_{l-s}\}
\end{equation}
such that $Q(e_i,e_j)=Q(f_i,f_j)=0$, $\sigma(e_i,f_j)=(-1)^{s-i}$, if $j=-i$, and $\sigma(e_i,f_j)=0$ if  $j\neq -i$. Define the grading on $V_{s;l}^{\mathfrak{so}}$ such that the $i$th component equal to the span of all vectors with index $i$ appearing in the tuple \eqref{tuple1}. It is an orthogonal grading. Then denote by $\delta_
{s;l}^{\mathfrak{so}}$ a degree $-1$ endomorphism of
$V_
{s;l}^{\mathfrak{so}}$ from the symplectic algebra such that $\delta_{s;l}^{\mathfrak{so}}(e_i)=e_{i-1}$ for $s-l+1\leq i\leq s$, $\delta_{s;l}^{\mathfrak{so}}(e_{s-l})=0$, $\delta_{s;l}^{\mathfrak{so}}(f_i)=f_{i-1}$ for $-s+1\leq i\leq l-s$, and
$\delta_{s;l}^{\mathfrak{so}}(f_{-s})=0$.

\item[\bf{(E2)}]
Given a nonnegative integer $m$
let
$\mathcal L_m^{\mathfrak{so}}=\displaystyle{\bigoplus_{i=-m}^m}E_i$
be an orthogonal graded spaces such that $\dim E_i=1$ for every admissible $i$ and let $\tau_m^{\mathfrak{so}}$ be a degree $-1$ endomorphism of $\mathcal L_m$ from the symplectic algebra which sends $E_i$ onto $E_{i-1}$ for every admissible $i$, except  $i=-m$, and  $\tau_m^{\mathfrak{so}}(E_{-m})=0$. In the case $\mathbb K=\mathbb R$ we also assume that $Q(\tau_m^{\mathfrak{sp}} (e), e)\geq 0$ for all $e_0\in E_{1}$.
\end{enumerate}

The following theorem gives the classification of all degree $-1$ endomorphisms from the orthogonal algebra of a graded orthogonal space $V$ and consequently the classification of all symbols  of curves (both of unparametrized and parametrized) of orthogonal flags of
$V$ with respect to $O(V)$ and $CO(V)$:

\begin{thm}
\label{soclass1}
Assume that $V$ is a graded orthogonal space.
\begin{enumerate}
\item
If the grading on $V$ is indexed by $\frac{1}{2}\mathbb Z_{\rm odd}$ then
a degree $-1$ endomorphism from $\mathfrak{so}(V)$
is conjugated (by an orthogonal transformation)  to the direct sum of endomorphisms of type $\delta_{s;l} ^{\mathfrak{so}}$, where $s\in \frac{1}{2}\mathbb Z_{\rm odd}$
and $0\leq l \leq 2s$. Moreover, for each pair of integers $(s,l)\in \frac{1}{2}\mathbb Z_{\rm {odd}}\times\mathbb Z$ with $0\leq l\leq 2s$ the number of appearances of $\delta_{s;l} ^{\mathfrak{so}}$ in this direct sum is an invariant of the conjugate class.
\item
If the grading on $V$ is indexed  by $Z$ and
$\mathbb K=\mathbb C$, then
a degree $-1$ endomorphism from $\mathfrak{so}(V)$
 is conjugated (by an orthonormal transformation) to the direct sum of endomorphisms of type $\delta_{s;l} ^{\mathfrak{so}}$ with a positive integer $s$
and  $0\leq l\leq 2s-1$, and of type  $\tau_m^{\mathfrak{so}}$ with a nonnegative integer $m$. Moreover, for each pair of integers $(s,l)$
with $0\leq l\leq 2s-1$  a nonnegative integer $m$ the numbers of appearances of $\delta_{s;l} ^{\mathfrak{so}}$ and $\tau_m ^{\mathfrak{so}}$ in this direct sum are  invariants of the conjugate class.

\item
If the grading on $V$ is indexed by  $\frac{1}{2}\mathbb Z_{\rm {odd}}$ and
$\mathbb K=\mathbb R$, then
a degree $-1$ endomorphism from $\mathfrak{so}(V)$
is conjugated  to the direct sum of endomorphisms of type $\delta_{s;l} ^{\mathfrak{so}}$, $\tau_m ^{\mathfrak{so}}$, and
$-\tau_m ^{\mathfrak{so}}$.
Moreover, for each pair of integers $(s,l)$
with $0\leq l\leq 2s-1$  and a nonnegarive integer $m$
the numbers of appearances of $\delta_{s;l} ^{\mathfrak{so}}$, $\tau_m ^{\mathfrak{so}}$, and
$-\tau_m ^{\mathfrak{so}}$  in this direct sum are  invariants of the conjugate class.
\end{enumerate}

\end{thm}

Theorem \ref{soclass1} shows that, up to a conjugation, the only orthogonally indecomposable endomorphisms are the endomorphisms of the type $\delta_{s;l}^{\mathfrak{so}}$ and $\tau_m$ for $\mathbb K=\mathbb C$ and of the type $\delta_{s;l}^{\mathfrak{so}}$, $\tau_m^{\mathfrak{so}}$, and $-\tau_m^{\mathfrak{so}}$ in the case $\mathbb K=\mathbb R$.

The proof of Theorem \ref{soclass1} is identical to the proof of Theorem \ref{spclass1}. The only difference is that in the present case the form $b$, defined by \eqref{formb},
is symmetric if $l$ is even and skew-symmetric if $l$ is odd. Therefore Lemma \ref{spclasscor} should be modified appropriately.

\begin{rem}
\label{quivrem}
{\rm This remark is about a possible relation of the problem of classification of symbols to the theory of quiver representations (\cite{gabriel}, \cite{quiv}).
In the case $G=GL(W)$( or $SL(W)$) there is an obvious one-to-one correspondence between the set of symbols (obtained in Theorem \ref{glclass}) and the set of indecomposable representations of the quivers with underlying indirect graph equal to the Dynkin diagram of type $A_\ell$ (if one does not take into account possible shift in the range of indices in the graded space). It would be interesting to link the obtained classification of indecomposable symplectic and orthogonal symbols with representations of quivers with the corresponding Dynkin diagrams as underlying indirect graphs.} $\Box$
\end{rem}
\section{Computation of algebraic prolongation of symbols for classical groups}
\setcounter{equation}{0}
\setcounter{thm}{0}
\setcounter{lem}{0}
\setcounter{prop}{0}
\setcounter{cor}{0}
\setcounter{rem}{0}
\setcounter{exm}{0}
\label{algprolongsec}

\subsection{Decomposition of the universal prolongation algebra} First we point out some general properties of the universal algebraic prolongation in the case when $\g$ is a graded reductive Lie algebra. Let $\mathfrak m$ be a line in $\g_{-1}$ and $\ug(\mathfrak m)$ be the universal algebraic prolongation of $\mathfrak m$, as defined in subsection \ref{univsect}.
Take $\delta\in\mathfrak m$.
According to Jacobson-Morozov  theorem \cite[Ch.III, Th. 17]{jacob} (see also \cite{vinberg2} for complex graded Lie algebras and \cite{doubproj} for real graded case), we can complete $\delta$ by elements $H$ and $Y$ of degree $0$ and $1$  respectively (in $\g$) to the standard basis of a $\sll_2$-subalgebra of $\mathfrak g$, i.e such that
\begin{equation}
\label{sl2}
[H,\delta] = 2\delta,\ [H,Y]=-2Y,\ [\delta,Y]=H.
\end{equation}

Let $\ng_{\rm{max}}(\mathfrak m)$ be the largest ideal in $\ug(\mathfrak m)$ concentrated in the non-negative degree (i.e., $\ng_{\rm{max}}(\mathfrak m)\subset\sum_{i\ge 0}\ug_i$). Such ideal exists since the sum of any two ideals concentrated in the non-negative degree will also be an ideal of this type. It is also clear that $\ng_{\rm{max}}(\mathfrak m)$ is graded, i.e. $\ng_{\rm{max}}(\mathfrak m)=\sum_{i}\bigl(\ng_{\rm{max}}(\mathfrak m)\bigr)_i$, where $\bigl(\ng_{\rm{max}}(\mathfrak m)\bigr)_i=\ng_{\rm{max}}(\mathfrak m)\cap\g_i$.

On the one hand,  since  $[\delta,\ng_{\rm{max}}(\mathfrak m)]\subset \ng_{\rm{max}}(\mathfrak m)$ and relations \eqref{sl2} hold, we get that $\sll_2\cap\ng _{\rm{max}}(\mathfrak m)=0$, which implies that
\begin{equation}
\label{geq3}
\dim \ug(\mathfrak m)/\ng_{\rm{max}}(\mathfrak m)\geq 3.
\end{equation}
On the other hand, under the identification of the algebra $\ug(\mathfrak m)$ with the algebra of infinitesimal symmetries of a flat curve $F_\mathfrak m$ with the symbol $\mathfrak m$ (see subsection \ref{univsect}), to any element of $\ug(\mathfrak m)$ we can assign the vector field on the curve $F_\mathfrak m$. Consider the subspace $\ng_{ne}(\mathfrak m)$ of $\ug(\mathfrak m)$, consisting of all elements of $\ug(\mathfrak m)$ for which the corresponding vector fields on $F_\mathfrak m$ are identically zero. Clearly, $\ng_{ne}(\mathfrak m)$ is an ideal of $\ug(\mathfrak m)$. This ideal is called the \emph{non-effectiveness ideal of $\ug(\mathfrak m)$}. Then the quotient algebra $\ug(\mathfrak m)/\ng_{ne}(\mathfrak m)$
can be realized as a finite-dimensional Lie algebra of
vector fields on a curve. From the classical Sophus Lie result it follows that
\begin{equation}
\label{leq3}
\dim  \ug(\mathfrak m)/\ng_{ne}(\mathfrak m)\leq 3
\end{equation}
 (see the original proof in \cite{slie} , its translation and
commentary in \cite{herm}, and a self-contained proof in the recent paper \cite{eastslov}).
Since $\ng_{ne}(\mathfrak m)\subset \mathfrak n_{\rm{max}}(\mathfrak m)$ we get from \eqref{geq3} and \eqref{leq3} that $\ng_{\rm{max}}(\mathfrak m)=\ng_{ne}(\mathfrak m)$ and that  $\dim \ug(\mathfrak m)/\ng_{\rm max}(\mathfrak m)=3$.
The latter implies the following
\begin{prop}
\label{nmaxlem}
$\ug(\mathfrak m)$ is a semidirect sum of the constructed embedding of $\sll_2$ into $\mg$ and $\mathfrak n_{\rm{max}}(\mathfrak m)$ ($=\mathfrak n_{ne}(\mathfrak m)$). In particular, \emph{$\ng_{\rm{max}}$ is an $\sll_2$-module}.
\end{prop}

The latter fact is very useful in the description of the universal algebraic prolongations of symbols of (unparametrized) curves of flags.

\begin{rem}
\label{parprolongrem}
{\rm For parametrized curves the corresponding universal prolongation of  a symbol $\delta\in \mg_{-1}$ is equal to  a semidirect sum of $\mathbb K\delta$ and the non-effectiveness ideal $\mathfrak n_{ne}$}.
\end{rem}

\subsection{The case of $G=GL(W)$ ($SL(W)$)}
\label{algproglv}
We say that a graded space $V$ which is also $\sll_2$-module is a \emph{nice $\sll_2$-module}, if the corresponding embedding of $\sll_2$ into $\gl(V)$ is spanned by endomorphisms of degree $-1$, $0$, and $1$. Let $V_1$ and $V_2$ be two
nice $\sll_2$-modules.
Then $\Hom(V_1,V_2)=V_2\otimes V_1^*$ is the $\sll_2$-module  and a graded space in a natural way.
Denote by $\mathfrak n(V_1,V_2)$ the maximal $\sll_2$-submodule of $\Hom(V_1,V_2)$  concentrated in the non-negative degree part.

Assume that $\mathfrak m=\mathbb R \delta$. By Theorem \ref{glclass} there exists a map $N:\{(r,s)\in \mathbb Z\times\mathbb Z: r\leq s<0\}\rightarrow \mathbb N\cup\{\infty\}$ with finite support such that $\delta$ is conjugated to the endomorphism $D_N$, which is  the direct sum of indecomposable symbols where  $\delta_{rs}$ is repeated $N(r,s)$ times.
The endomorphism $D_n$ acts on the space
\begin{equation}
\label{Vn}
\mathcal V_N
= \displaystyle{\bigoplus_{r\leq s<0} V_{rs}\otimes \mathbb K^{N(r,s)}}.
\end{equation}
First, $D_N$ can be extended to a subalgebra of $\gl(\mathcal V_N)$ isomorphic
to $\sll_2$ which acts irreducibly on each $V_{rs}$.
From \eqref{Vn} it follows that
$$\gl(\mathcal V_N)=\mathcal V_N\otimes(\mathcal V_N)^*=
\displaystyle{\bigoplus_{\tiny{\begin{array}{l}r_1\leq s_1<0,\\ r_2\leq s_2<0\end{array}}}} \Hom (V_{r_1s_1}, V_{r_2 s_2})\otimes \Hom(\mathbb K^{N(r_1,s_1)},\mathbb K^{N(r_2,s_2)})
.$$ Second, by definition of $\ng(V_{r_1s_1},V_{r_2s_2})$ we have that
\begin{equation}
\label{posprel}
\mathfrak p=\displaystyle{\bigoplus_{\tiny{\begin{array}{l} r_1\leq s_1<0, \\ r_2\leq s_2<0\end{array}}}} \ng(V_{r_1s_1}, V_{r_2s_2})\otimes  \Hom(\mathbb K^{N(r_1,s_1)},\mathbb K^{N(r_2,s_2)})
\end{equation}
is the maximal $\sll_2$-module in $\gl(\mathcal V_N)$ concentrate in the non-negative degree part of  $\gl(\mathcal V_N)$.
Besides, $\mathfrak p$ is a subalgebra of $\gl(\mathcal V_N)$. Indeed, $[\mathfrak p,\mathfrak p]$ is an $\sll_2$-module  in $\gl(\mathcal V_N)$ concentrate in the non-negative degree, which implies that $[\mathfrak p,\mathfrak p]\subset \mathfrak p$, because $\mathfrak p$ is the maximal  $\sll_2$-module satisfying this property.  Therefore $p=\ng_{\rm max}(\mathbb K D_n)$.

It remains to describe   $\ng(V_{r_1,s_1},V_{r_2,s_2})$ more explicitly.
Set $l_i=s_i-r_i$, $i=1,2$.
Note that $\sll_2$-submodule $V_{r_2s_2}\otimes (V_{r_1s_1})^*$ is decomposed into the irreducible $\sll_2$-modules as follows:
\begin{equation}
\label{Vrsrs}
\Hom(V_{r_1s_1}, V_{r_2s_2})=V_{r_2s_2}\otimes (V_{r_1s_1})^*\cong\bigoplus_{i=0}^{\min\{
l_1,l_2\}}\Pi_{l_1+l_2-2i},
\end{equation}
where $\Pi_j$ denotes an irreducible $\sll_2$-module of dimension $j+1$ (see, for example, \cite{Fulton}). Moreover, the submodule
$\Pi_{l_1+l_2}$
of the largest dimension  is generated by the elements of highest (or lowest) degree in
$V_{r_2s_2}\otimes (V_{r_1s_1})^*$, which is equal to $s_2-r_1$ ($r_2-s_1$ respectively). The range of degrees for each next submodule in this decomposition is shrunk by $1$ from both left and right sides, i.e. the submodule $\Pi_{l_1+l_2-2i}$ has degrees varying from $r_2-s_1+i$ to  $s_2-r_1-i$. The submodule $\ng(V_{r_1s_1},V_{r_2s_2})$ is equal to the direct sum of submodules from the decomposition \eqref{Vrsrs} for which all degrees are non-negative. Therefore,
\begin{equation}
\label{ngrsrs}
\ng(V_{r_1s_1},V_{r_2s_2})\cong\bigoplus_{i=\max\{0, s_1-r_2\}}^{\min\{
l_1,l_2, s_2-r_1\}}\Pi_{l_1+l_2-2i}.
\end{equation}
In particular,
\begin{eqnarray}
&~&\label{rseq}\ng(V_{rs},V_{rs})\cong\Pi_0\cong\mathbb K\,\rm{Id},
\\
&~&\label{rsgeq}\ng(V_{r_1s_1},V_{r_2s_2})=0 \text{ if and only if } s_2<s_1 \text{ or } r_2<r_1.
\end{eqnarray}

Let us prove statement \eqref{rsgeq}. Indeed by \eqref{ngrsrs} $\ng(V_{r_1s_1},V_   {r_2s_2})=0$ if and  only if one of the following three conditions holds:
\begin{enumerate}
\item $s_1-r_2>s_1-r_1\equiv r_2<r_1$;
\item $s_1-r_2>s_2-r_2\equiv s_2<s_1$;
\item $s_1-r_2>s_2-r_1$,
\end{enumerate}
which proves the ``if'' part of \eqref{rsgeq}.
Further, if conditions (1) and (2) does not hold then the condition (3) does not hold as well, which proves the ``only if''  part.

Further, it is clear that
\begin{equation}
\label{tensorgl}
\mathfrak n(V_{r_1,s_1}\otimes\mathbb K^{N_1},V_{r_2,s_2}\otimes\mathbb K^{N_2})\cong
\mathfrak n(V_{r_1,s_1}, V_{r_2,s_2})\otimes {\rm Hom}\Bigl(\mathbb K^{N(r_1,s_1)},
 \mathbb K^{N(r_2,s_2)}\Bigr).
\end{equation}
Combining \eqref{posprel}, \eqref{rseq}, \eqref{rsgeq}, \eqref{tensorgl}, and Proposition  \ref{nmaxlem} we get the following
\begin{thm}
\label{prolongN}
The universal prolongation algebra $\ug(\mathbb K D_N)$ of the symbol $\mathbb K D_N$
is equal to the semidirect sum of the constructed embedding of $\sll_2$ into $\gl(V)$ and
$$
\bigoplus_{\tiny{\begin{array} {c} s_1\leq s_2,\, r_1\leq r_2,\\ r_j\leq s_j<0\end{array}}}\ng(V_{r_1s_1},V_{r_2s_2})\otimes {\rm Hom}\Bigl(\mathbb K^{N(r_1,s_1)},
 \mathbb K^{N(r_2,s_2)}\Bigr), 
$$
where  $\ng(V_{r_1s_1},V_{r_2s_2})$ is as in \eqref{ngrsrs}.
\end{thm}

In particular, if $\delta=\delta_{rs}$ then by \eqref{rseq} one has $\mathfrak n_{\rm{max}}(\delta)=\mathbb K\,{\rm Id}$
and $\mathfrak u(\delta)\cong \gl_2$. This can be applied to Example \ref{exm2} above.

\begin{rem}
\rm{
It is clear that in the case of $G=SL(W)$ the universal prolongation of a symbol consists of the traceless part of the universal prolongation of the same symbol for $G=GL(W)$.}
\end{rem}

\subsection{The case of $G=Sp(W)$ ($CSp(W)$) and $G=O(W)$  ($CO(W)$)}
First note that
in the case $G=CSp(W)$ ($G=CO(W)$) the universal algebraic prolongation of a symbol is equal to the direct sum of the universal algebraic prolongation of the same symbol for $G=Sp(W)$ ($G=O(W)$) with $\mathbb K$. Therefore it is sufficient to concentrate on the case $G=Sp(W)$ and $G=O(W)$.
We will primary treat the symplectic case and briefly indicate what changes should be done for the orthogonal case.
Let $V$ be a graded symplectic
space
 which is also a nice $\sll_2$-module such that the corresponding embedding of $\sll_2$ into $\gl(V)$ belongs to $\mathfrak {sp}(V)$
In this case we will say that $V$ is a \emph{nice symplectic
$\sll_2$-module}.
The symplectic algebra $\mathfrak{sp}(V)$
is a $\sll_2$-module and  a graded space in a natural way. Denote by $\mathfrak l^{\mathfrak{sp}}(V)$
the maximal $\sll_2$-submodule of $\mathfrak {sp}(V)$
concentrated in the non-negative degree part. From the maximality assumption it follows  that $\mathfrak l^{\mathfrak{sp}}(V)$
is a subalgebra of $\mathfrak {sp}(V)$.

Further, assume that
$V=\displaystyle{\bigoplus_{i=1}^q L_i}$,
where each $L_i$ is a graded symplectic
space and nice symplectic
$\sll_2$-modules (with all structures induced from $V$).
The restriction $\sigma|_{L_i}$ to $L_i$ of the symplectic form $\sigma$ of $V$
defines natural identification between $L_i$ and its dual space $L_i^*$.
Here $\sigma$ denotes the symplectic form on $V$
Consider the following splitting
of $\gl(V)$:
\begin{equation}
\label{glsplit}
\gl(V)=\bigoplus_{i,j=1}^q\Hom(L_i, L_j).
\end{equation}
An endomorphism  $A\in\gl(V)$, having the decomposition $A=\sum_{i,j=1}^q A_{ij}$ with respect to the splitting \eqref{glsplit},
belongs to $\mathfrak{sp}(V)$
if and only if $A_{ii}\in \mathfrak{sp}(L_i)$ for all $1\leq i\leq q$ and $A_{ij}=-A_{ji}^*$ for all $1\leq i\neq j\leq q$ (here the dual linear map
$A_{ji}^*$ is considered as a map from $L_i$ to $L_j$ under the aforementioned identification $L_i\sim L_i^*$ and $L_j\sim L_j^*$).
Therefore the map $A\mapsto \sum_{i=1}^q A_{ii}+\sum_{1\leq i<j\leq q}A_{ij}$ defines the identification
\begin{equation}
\label{idsp2}
\mathfrak{sp} \left(\bigoplus_{i=1}^q L_i\right)\cong \bigoplus_{i=1}^q \mathfrak{sp}(L_i)\oplus\bigoplus_{1\leq i<j\leq q}\Hom(L_i,L_j).
\end{equation}
Moreover,
 \begin{equation}
\label{idsp2pos}
\mathfrak{l}^{\mathfrak{sp}} \left(\bigoplus_{i=1}^q L_i\right)\cong \bigoplus_{i=1}^q \mathfrak{l}^{\mathfrak{sp}}(L_i)\oplus\bigoplus_{1\leq i<j\leq q}\mathfrak n(L_i,L_j),
\end{equation}
where $\mathfrak n(L_i,L_j)$ is as in the previous subsection.

Now take a symplectic 
symbol $\mathfrak m=\mathbb K\delta$.
Let $V_{s;l}^{\mathfrak{sp}}$ and $\mathcal L_m^{\mathfrak{sp}}$.
are graded symplectic
spaces as in items (D1) and (D2) of subsection \ref{symbsec}.
According to Theorem \ref{spclass1}
 $\delta$ is conjugated to
a direct sum of endomorphism of types $\delta_{s;l}^{\mathfrak{sp}}$, $\tau_m^{\mathfrak{sp}}$, and , in the case of $\mathbb K=\mathbb R$, also of type $-\tau_m^{\mathfrak{sp}}$.
Therefore one can always assume that $V=\displaystyle{\bigoplus_{i=1}^q }L_i$, where each $L_i$ is either equal to $V_{s;l}^{\mathfrak{sp}}$ or $\mathcal L_m^{\mathfrak{sp}}$.
In the symplectic case the endomorphism $\delta$ can be extended to a subalgebra of $\mathfrak {sp}\left(\displaystyle{\bigoplus_{i=1}^q L_i}\right)$ isomorphic
to $\sll_2$ such that if $L_i=V_{s;l}^{\mathfrak{sp}}$ then $V_{s;l}^{\mathfrak{sp}}$ (with respect to the induced action) is the sum of two irreducible $\sll_2$-submodules
\begin{equation}
\label{EFrs}
E_{s;l}
={\rm span} \{e_i\}_{s-l\leq i\leq s},\quad F_{s;l}
={\rm span} \{f_i\}_{-s\leq i\leq l-s},
\end{equation}
where $e_i$ and $f_i$ are as in \eqref{tuple1}, and if $L_i= \mathcal L_m^{\mathfrak{sp}}$, then
$\mathcal L_m^{\mathfrak{sp}}$ (with respect to the induced action) is an irreducible $\sll_2$-module.
By analogy with the previous subsection \emph{the universal algebraic prolongation $\mathfrak u(\mathbb R \delta)$ of the symbol $\mathbb R \delta$ is equal to the semidirect sum of the constructed  embedding of $\sll_2$ into $\mathfrak{sp}(V)$ 
and the algebra $\mathfrak l^{\mathfrak{sp}} \left(\bigoplus_{i=1}^q L_i\right)$}.

The orthogonal case is treated in completely the same way. We only need to replace everywhere from the beginning of this subsection the word ``symplectic''  by the word `` orthogonal", the sign $\mathfrak {sp}$ by the sign $\mathfrak {so}$, and the symplectic form $\sigma$ by a non-degenerate symmetric form $Q$.

By identification \eqref{idsp2pos} and the analogous formula for the orthogonal case
in order to compute
$\mathfrak l^{\mathfrak{sp}} \left(\bigoplus_{i=1}^q L_i\right)$
($\mathfrak l^{\mathfrak{so}}\left(\bigoplus_{i=1}^q L_i\right)$
and, consequently, the universal prolongation $\mathfrak u(\mathbb R \delta)$ in $\mathfrak{sp}(V)$ ($\mathfrak{so}(V)$) it is sufficient to compute spaces $\mathfrak l^{\mathfrak{sp}}(V)$
($\mathfrak l^{\mathfrak{spo}}(V)$)
and $\Hom(V_1, V_2)$, where $V$, $V_1$, and $V_2$ are equal
either to $V_{l;s}^{\mathfrak{sp}}$ or $\mathcal L_m^{\mathfrak{sp}}$
(either to $V_{l;s}^{\mathfrak{so}}$ or $\mathcal L_m^{\mathfrak{so}}$).

For this first denote by $\mathfrak n_{(r_1, s_1),(r_2, s_2)}$  the righthand side of \eqref{ngrsrs},
\begin{equation}
\label{sumpi}
\mathfrak n_{(r_1, s_1),(r_2, s_2)}= \bigoplus_{i=\max\{0, s_1-r_2\}}^{\min\{
l_1,l_2, s_2-r_1\}}\Pi_{l_1+l_2-2i},
\end{equation}
where $l_i=s_i-r_i$.

\subsubsection{Description of $\mathfrak l^{\mathfrak{sp}}(V_{s;l}^{\mathfrak{sp}})$ and $\mathfrak l^{\mathfrak{so}}(V_{s;l}^{\mathfrak{so}})$}
\label{case1}
\indent
\medskip

First let us give a convenient description of the algebras $\mathfrak
{sp}\Bigl(V_{s;l}^{\mathfrak{sp}}\Bigr)$ and $\mathfrak
{so}\Bigl(V_{s;l}^{\mathfrak{so}}\Bigr)$.
In the sequel we will use the same notation for $E_{s;l}
$ and $F_{s;l}
$ in the orthogonal case, as in \eqref{EFrs}.
The symplectic form or  the non-degenerate symmetric form define the following natural identifications:
 $E_{s;l}
\cong (F_{s;l}
)^*$ and
$F_{s;l}
\cong(E_{s;l}
)^*$.
Keeping in mind these identifications, an endomorphism $A\in\gl(V)$, having decomposition $A=A_{11}+A_{22}+A_{12}+A_{21}$ with respect to the splitting $$\gl(V_{s;l}^{\mathfrak{sp}})=\gl(E_{s;l}
)\oplus\gl(F_{s;l}
)\oplus\Hom(F_{s;l}
,E_{s;l}
)\oplus \Hom(E_{s;l}
,F_{s;l}
),$$ belongs to $\mathfrak{sp}\Bigl(V_{s;l}^{\mathfrak{sp}}\Bigr)$ if and only if $A_{12}^*=A_{12}$, $A_{21}^*=A_{21}$, and $A_{22}=-A_{11}^*$. It belongs to $\mathfrak{so}\Bigl(V_{s;l}^{\mathfrak{so}}\Bigr)$ if and only if $A_{12}^*=-A_{12}$, $A_{21}^*=-A_{21}$, and $A_{22}=-A_{11}^*$. Therefore the map $A\mapsto A_{11}+A_{22}+A_{12}$ defines the following identifications
\begin{eqnarray}
&~&\label{idsp1}
\mathfrak{sp}(V_{s;l}^{\mathfrak{sp}})\cong \gl(E_{s;l}
)\oplus S^2(E_{s;l}
)\oplus S^2 (F_{s;l}
),\\
&~&\label{idso1}
\mathfrak{so}(V_{s;l}^{\mathfrak{so}})\cong \gl(E_{s;l}
)\oplus S^2(E_{s;l}
)\oplus \wedge^2 (F_{s;l}
),
\end{eqnarray}
where  $S^2(E_{s;l}
)$ and $S^2(F_{s;l}
)$ denote the symmetric square of $E_{s;l}
$ and $F_{s;l}
$, respectively, and $\wedge^2(E_{s;l}
)$ and $\wedge^2(F_{s;l}
)$ denote the skew-symmetric square of $E_{s;l}
$ and $F_{s;l}
$.

Here $S^2(E_{s;l}
)$ and $\wedge^2(E_{s;l}
)$ are subspaces of
$\Hom(F_{s;l}
, E_{s;l}
)$; $S^2(F_{s;l}
)$ and $\wedge^2(F_{s;l}
)$ are subspaces of $\Hom(E_{s;l}
, F_{s;l}
)$.
Keeping this in mind, we define $\mathfrak{l}_{s;l}^{\mathfrak{sp},1}$ ($\mathfrak{l}_{s;l}^{\mathfrak{so},1}$) as the maximal $\sll_2$-submodule of
$S^2(E_{s;l}
)$ ($\wedge^2(E_{s;l}
)$)
concentrated in the non-negative degree part of $S^2(E_{s;l}
)$ ($\wedge^2(E_{s;l}
)$).
Similarly, let $\mathfrak{l}_{s;l}^{\mathfrak{sp},2}$ ($\mathfrak{l}_{s;l}^{\mathfrak{so},2}$) be the maximal $\sll_2$-submodule of
$S^2(F_{s;l}
)$ ($\wedge^2(E_{s;l}
)$).
Then from \eqref{idsp1}
\begin{eqnarray}
&~&\label{lrse1}
\mathfrak{l}^{\mathfrak sp}(V_{s;l}^{\mathfrak{sp}})\cong \mathfrak n_{(s-l,s),(s-l,s)}\oplus \mathfrak{l}_{s;l}^{\mathfrak{sp},1}\oplus \mathfrak{l}_{s;l}^{\mathfrak{sp},2}\cong\mathbb K\oplus \mathfrak{l}_{s;l}^{\mathfrak{sp},1}\oplus \mathfrak{l}_{s;l}^{\mathfrak{sp},2}\\
&~&\label{krse1}
\mathfrak{l}^{\mathfrak{so}}(V_{s;l}^{\mathfrak{so}})\cong
\mathfrak n_{(s-l,s),(s-l,s)}\oplus \mathfrak{l}_{s;l}^{\mathfrak{so},1}\oplus \mathfrak{l}_{s;l}^{\mathfrak{so},2}\cong\mathbb K\oplus \mathfrak{l}_{s;l}^{\mathfrak{so},1}\oplus \mathfrak{l}_{s;l}^{\mathfrak{so},2}
\end{eqnarray}

Let us describe   $\mathfrak{l}_{s;l}^{\mathfrak{sp},1}
$, $\mathfrak{l}_
{s;l}^{\mathfrak{sp},2}
$
$\mathfrak{l}_
{s;l}^{\mathfrak{so},1}
$ and
$\mathfrak{l}_
{s;l}^{\mathfrak{so},2}
$.
Let $l=s-r$. In order to describe $\mathfrak{l}_
{s;l}^{\mathfrak{sp},1}$ and $\mathfrak{l}_
{s;l}^{\mathfrak{so},1}
$
note that $\sll_2$-submodules $S^2(E_{s;l}
)$ and $\wedge^2(E_{s;l}
)$) are decomposed into the irreducible $\sll_2$-modules as follows:
\begin{eqnarray}
&~&S^2(E_{s;l}
)=\bigoplus_{i=0}^{
[\frac{l}{2}]}\Pi_{2l-4i}
\label{Vrseps};\\
&~&\label{Mrseps}
\wedge^2(E_{s;l}
)=\bigoplus_{i=0}
^{[\frac{l-1}{2}]}\Pi_{2l-2-4i}.
\end{eqnarray}
(see, for example, \cite{Fulton}).

The submodule
$\Pi_{2l}$
of the largest dimension in $S^2(E_{s;l}
)$ is generated by the elements of highest or lowest degree in
$S^2(E_{s;l}
)$, which are equal to $2s$ and $2s-2l$, respectively. The range of degrees for each next submodule in the decomposition \eqref{Vrseps} is shrunk by $2$ from both left and right sides, i.e. the submodule $\Pi_{2l-4i}$ has degrees varying from 
$2s-2l+2i$ and $2s-2i$. The submodule $\mathfrak{l}_{s;l}^{\mathfrak{sp},1}$ is equal to the direct sum of the submodules from the decomposition \eqref{Vrseps} for which all degrees are non-negative, i.e. for which
$\max\{0, l-[s]\}\leq i\leq \min\{
[\frac{l}{2}],
[s]\}$.
Since $l\leq 2[s]$, we have that $[\frac{l}{2}]\leq [s]$.
Therefore

\begin{equation}
\label{ngrseps}
\mathfrak{l}_{s;l}^{\mathfrak{sp},1}=\bigoplus_{i=\max\{0,
l-[s]\}}^{
[\frac{l}{2}]}\Pi_{2l-4i},
\end{equation}
Note that from the condition 
$l\leq 2[s]$
and the last formula it follows that  $\mathfrak{l}_{s;l}^{\mathfrak{sp},1}\neq 0$.

In order to get  $\mathfrak{l}_{l;s}^{\mathfrak{sp},2}$ we have to replace $s$
by $l-s$
in the righthand side of \eqref{ngrseps}, i.e.  $$\mathfrak{l}_{l;s}^{\mathfrak{sp},2}=\bigoplus_{i=\max\{0, l-[l-s]]\}}^
{[\frac{l}{2}]}\Pi_{2l-4i},$$
Again from the fact that
$l<2[s]$ it follows that
\begin{equation}
\label{ngrseps2}
\mathfrak{l}_{rs}^{\varepsilon,2}=\begin{cases} \mathbb K&\text{if \text{$l$ is even and }}
s=\frac{l}{2},\\
0&\text{otherwise}.
\end{cases}
\end{equation}
Besides, by \eqref{rseq} one has  $\mathfrak n_{(s-l,s),(s-l,s)}=\mathbb K$. Substituting this and relation \eqref{ngrseps2}
into \eqref{ngrseps} we get that in the case  $l$ is odd or $s\neq \frac{l}{2}$
\begin{equation}
\label{lrse1fin}
\mathfrak{l}(V_{s;l}^{\mathfrak{sp}})\cong \mathbb K\oplus \mathfrak{l}_{s;l}^{\mathfrak{sp},1},
\end{equation}
where $\mathfrak{l}_{s;l}^{\mathfrak{sp},1}$ is as in  \eqref{ngrseps}.
Finally, by  \eqref{ngrseps} we get $\mathfrak{l}_{p;2p}^{\mathfrak{sp},1}=\mathbb K$ for any nonnegative integer $p$. From this, formula  \eqref{ngrseps2}, and identification
\eqref{lrse1}  it is easy to see that for any nonnegative integer $p$
\begin{equation}
\label{lrse1spec}
 \mathfrak l^{\mathfrak {sp}}(V_{p;2p}^{\mathfrak{sp}})\cong \sll_2.
\end{equation}

Similarly in the orthogonal case the submodule
$\Pi_{2l-2}$
of the largest dimension in $\wedge^2(E_{s;l}
)$ is generated by the elements of highest or lowest degree in
$\wedge^2(E_{s;l}
)$, which is equal to $2s-1$ and $2s-2l+1$ respectively. The range of degrees for each next submodule in the decomposition \eqref{Mrseps} is shrunk by $2$ from both left and right sides, i.e. the submodule $\Pi_{2l-2-4i}$ has degrees varying from $2s-2l+1+2i$ to $2s-1-2i$. The submodule $\mathfrak{l}_{s;l}^{\mathfrak{so},1}$ is equal to the direct sum of submodules from the decomposition \eqref{Mrseps} for which all degrees are non-negative, i.e. for which $\max\{0,
l-[s+\frac{1}{2}]
\}\leq i\leq \min\{
[\frac{l-1}{2}],
[s-\frac{1}{2}]
\}$. Since $l\leq 2(s+\{s\})-1$, we have that $[\frac{l-1}{2}]\leq
[s-\frac{1}{2}]$. Thus

\begin{equation}
\label{ngkrseps}
\mathfrak{l}_{s;l}^{\mathfrak{so},1}=\bigoplus_{i=\max\{0,l-[s+\frac{1}{2}] ]\}}^{
[\frac{l-1}{2}]}\Pi_{2l-2-4i},
\end{equation}
Note that from the condition
$l\leq 2(s+\{s\})-1$
and the last formula it follows that  $
\mathfrak{l}_{s;l}^{\mathfrak{so},1}
\neq 0$.

In order to get  $
\mathfrak{l}_{s;l}^{\mathfrak{so},2}
$ we have to replace  $s$
by $l-s$
in the righthand side of \eqref{ngkrseps}, i.e.
$$\mathfrak{l}_{s;l}^{\mathfrak{so},2}
=\bigoplus_{i=\max\{0, l-[l-s+\frac{1}{2}]\}}^
{[\frac{l-1}{2}]}\Pi_{2l-2-4i},$$
Again from the fact that  $l\leq 2(s+\{s\})-1
$ it follows that
\begin{equation}
\label{ngkrseps2}
\mathfrak{l}_{s;l}^{\mathfrak{so},2}=\begin{cases} \mathbb K&\text
{$l$ is odd and }
s=\frac{l}{2},\\
0&\text{otherwise}.
\end{cases}
\end{equation}
Similarly to \eqref{lrse1fin}
in the case  $l$ is even or $s\neq \frac{l}{2}$ we get that
\begin{equation}
\label{krse1fin}
\mathfrak{l}^{\mathfrak{so}}(V_{s;l}^{\mathfrak{so}})
\cong \mathbb K\oplus
\mathfrak{l}_{s;l}^{\mathfrak{so},1},
\end{equation}
where $\mathfrak{l}_{s;l}^{\mathfrak{so},1}$ is as in  \eqref{ngkrseps}.
Finally, by  \eqref{ngkrseps} we get $\mathfrak{l}_{p-\frac{1}{2};2p-1}^{\mathfrak{so},1}=\mathbb K$ for any $p\in\mathbb N$. From this, formula  \eqref{ngkrseps2}, and identification
\eqref{krse1}  it is easy to see that
\begin{equation}
\label{krse1spec}
\mathfrak{l}^{\mathfrak{so}}(V_{p-\frac{1}{2};2p-1}^{\mathfrak{so}})
\cong \sll_2,\quad p\in\mathbb N.
\end{equation}

\subsubsection{Description of $\mathfrak l^{\mathfrak{sp}}(\mathcal L_m^{\mathfrak{sp}})$ and $\mathfrak l^{\mathfrak{so}}(\mathcal L_m^{\mathfrak{so}})$}
\label{case2}
\indent
\medskip

Note that $\mathfrak l^{\mathfrak{sp}}(\mathcal L_m^{\mathfrak{sp}})\subset \mathfrak n(\mathcal L_m^{\mathfrak{sp}}, \mathcal L_m^{\mathfrak{sp}})\cong \mathbb K {\rm Id}$, but $\rm {Id}\notin \mathfrak{sp}(\mathcal L_m^{\mathfrak{sp}})$. In the same way $\mathfrak l^{\mathfrak{so}}(\mathcal L_m^{\mathfrak{so}}
)\subset \mathfrak n(\mathcal L_m^{\mathfrak{so}}, \mathcal L_m^{\mathfrak{so}})\cong \mathbb K {\rm Id}$, but $\rm {Id}\notin \mathfrak{so}(\mathcal L_m^{\mathfrak{so}})$
Therefore
\begin{equation}
 \label{splm}
 \mathfrak l^{\mathfrak{sp}}(\mathcal L_m^{\mathfrak{sp}})=0, \quad \mathfrak l^{\mathfrak{so}}(\mathcal L_m^{\mathfrak{so}})=0.
 \end{equation}

\subsubsection{Description of $\mathfrak n(V_{s_1;l_1}^\mathfrak{sp}, V_{s_2; l_2}^\mathfrak{sp})$ (or
of $\mathfrak n(V_{s_1;l_1}^\mathfrak{so}, V_{s_2; l_2}^\mathfrak{so}
)$)}
\label{case3}
\indent
\medskip

By Remark \ref{halfintrem} we are interested only in the cases when both $s_1$ and $s_2$ are integers or both $s_1$ and $s_2$ belong to $\frac{1}{2}\mathbb Z_{\rm{odd}}$. Set $r_1=s_1-l_1$ and $r_2=s_2-l_2$. In the sequel $\varepsilon$ will mean other $\mathfrak{sp}$ or $\mathfrak{so}$.
Note that
\begin{equation*}
\Hom(V_{s_1;l_1}
^\varepsilon, V_{s_2; l_2}^\varepsilon
)=\Hom( E_{s_1;l_1}
, E_
{s_2;l_2}
)\oplus
\Hom( F_{s_1;l_1}
, F_
{s_2;l_2}
)
\oplus
\Hom( E_{s_1;l_1}
, F_
{s_2;l_2}
)
\oplus
\Hom( F_{s_1;l_1}
, E_
{s_2;l_2}
).
\end{equation*}
Therefore,
\begin{equation}
\label{nr1s1r2s2}
\mathfrak n(V_{s_1;l_1}^{\varepsilon}, V_{s_2; l_2}^\varepsilon)
\cong \ng_{(r_1,s_1),(r_2,s_2)}\oplus \ng_{(-s_1 ,-r_1),(-s_2,-r_2)}
\oplus \ng_{(r_1,s_1),(-s_2,-r_2)}\oplus \ng_{(-s_1,-r_1),(r_2,s_2)}
\end{equation}
Often many terms in the righthand side of \eqref{nr1s1r2s2} vanish.
First, let us analyze the term $\ng_{(r_1,s_1),(-s_2,-r_2)}$. By \eqref{rsgeq} if 
$\ng_{(r_1,s_1),(-s_2,-r_2)}
\neq 0$, then $r_1\leq -s_2$ and $s_1\leq -r_2$ or, equivalently, $s_1-l_1\leq -s_2$ and $s_1\leq l_2-s_2$.
In the symplectic case, in addition we have  $l_1\leq 2[s_1]$ and $l_2\leq 2[s_2]$. Using all four inequalities we get that $s_1=s_2$, $s_1,s_2$ are integers and $l_1=l_2=2s_1$. This together with \eqref{rseq} implies that
\begin{equation}
\label{ng3}
\ng_{(r_1,s_1),(-s_2,-r_2)}=\ng_{(s_1-l_1,s_1),(-s_2,l_2-s_2)}=\begin{cases} \mathbb K&  \text{if } s_1=s_2,\,\, s_1, s_2
\in\mathbb Z,\,\, l_1=l_2=2s_1\\
0&\text{otherwise}.
\end{cases}
\end{equation}
In the orthogonal case, in addition to $s_1-l_1\leq -s_2$ and $s_1\leq l_2-s_2$ we have  $l_1\leq 2(s_1+\{s_1\})-1$ and
$l_2\leq 2(s_2+{s_2})-1]$. Using all four inequalities we get that $s_1=s_2$,  $s_1,s_2$ belongs to $\frac{1}{2}\mathbb Z_{\rm{odd}}$, and $l_1=l_2=2s_1$. So, we will have the same formula as \eqref{ng3} except that $s_1, s_2$ belongs to $\frac{1}{2}\mathbb Z_{\rm{odd}}$.

Further, by \eqref{rsgeq}
\begin{equation}
\label{01}
\ng_{(-s_1,-r_1),(-s_2,-r_2)}=0 \text{ if and only if } r_1<r_2 \text{ or } s_1<s_2.
\end{equation}
Substituting formulas \eqref{ng3} (or its analog for the orthogonal case), \eqref{rsgeq}, \eqref{01} into \eqref{nr1s1r2s2}  and using \eqref{sumpi} for nonzero terms, one gets an explicit expression
for $\mathfrak n(V_{s_1;l_1}^\mathfrak{sp}, V_{s_2; l_2}^\varepsilon)$
. In particular, if $(r_1,s_1,r_2,s_2)\neq
(-s,s, -s,s)$, then at least two of the first three terms in  the right handside of \eqref{nr1s1r2s2} vanish. Finally, it is easy to see that
\begin{equation}
\label{ngrec}
\ng(V_{s;2s}^\varepsilon, V_{s; 2s}^\varepsilon)\cong \mathfrak{gl}_2, \quad \varepsilon=\mathfrak{sp}\text{ or } \mathfrak{so}.
\end{equation}


\subsubsection{Description of $\mathfrak n(V_{s;l}^{\mathfrak{sp}}, \mathcal L_m^{\mathfrak{sp}})$
and of $\mathfrak n(V_{s;l}^{\mathfrak{so}}, \mathcal L_m^{\mathfrak{so}})$}
\label{case4}
\indent
\medskip

By analogy with the previous case
\begin{equation}
\label{nrslm}
\mathfrak n(V_{s;l}^
\varepsilon, \mathcal L_m^
\varepsilon)=\ng_{(s-l,s),(-m,m)}+\ng_{(-s,l-s),(-m,m)}
\end{equation}
First, let us analyze the term  $\ng_{(s-l,s),(-m,m)}$. By \eqref{rsgeq} if $\ng_{(s-l,s),(-m,m)}\neq 0$, then $s-l\leq -m$ and $s\leq m$.  By Remark \ref{halfintrem}, $s\in \frac{1}{2}\mathbb Z_{\rm{odd}}$ in the symplectic case and $s\in\mathbb Z$ in the orthogonal case,  and therefore $l\leq 2s-1$ in both cases,
which together with the previous inequalities implies that $-s+1\leq s-l\leq -m\leq -s$, which is impossible.
Consequently, for any admissible triple $(l,s,m)$
we get
$\ng_{(s-l,s),(-m,m)}=0,$
i.e.
\begin{equation}
\label{ngl}
\mathfrak n(V_{s;l}^{\varepsilon}, \mathcal L_m)=\ng_{(-s,l-s),(-m,m)}, \quad \varepsilon=\mathfrak{so} \text{ or } \mathfrak {sp}.
\end{equation}
Besides, $\ng_{(-s,l-s),(-m,m)} \neq 0$ if and only if $l-s\leq m\leq s$
and it can be computed using formula \eqref{sumpi}.


\subsubsection{Description of $\mathfrak n(\mathcal L_{m_1}^{\rm{sp}}, \mathcal L_{m_2}^{\rm{sp}})$ (or of $\mathfrak n(\mathcal L_{m_1}^{\rm{so}}, \mathcal L_{m_2}^{\rm{so}})$)}
\label{case5}
\indent
\medskip

By \eqref{rsgeq}, both for $\varepsilon=\mathfrak{sp}$ and $\varepsilon=\mathfrak{so}$ we have

\begin{equation}
\label{ll}
\mathfrak n(\mathcal L_{m_1}^\varepsilon, \mathcal L_{m_2}^\varepsilon)=\ng_{(-m_1,m_1), (-m_2,m_2)} =\begin{cases} \mathbb K&  \text{if }m_1=m_2\\
0&  \text{if }m_1\neq m_2.
\end{cases}
\end{equation}
\subsubsection{The case of tensor products}
\label{case6}
\indent
\medskip
\label{condGsubsec}

In order to use formula \eqref{idsp2pos} more effectively it is worth also to say more about the space $\mathfrak l^{\mathfrak{sp}}(V\otimes \mathbb K^N)$, where $V$ is a  graded symplectic space which is also a nice symplectic $\sll_2$-module and   $V\otimes \mathbb K^N$ inherits from $V$ the structure of a graded symplectic space and of a
nice symplectic  $\sll_2$-module in a natural way. From \eqref{idsp2} it is easy to get the following natural identification:

\begin{equation}
\label{idspkN}
\mathfrak{sp}(V\otimes\mathbb K^N)\cong (\mathfrak{sp}(V)\otimes \mathbb K^N)\oplus(\mathfrak{gl}(V)\otimes\wedge^2\mathbb K^N),
\end{equation}
where $\wedge^2K^N$ denotes a skew-symmetric square of $\mathbb K^N$.
Consequently,
 \begin{equation}
\label{idspkNpos}
\mathfrak{l}^\mathfrak{sp}(V\otimes\mathbb K^N)\cong (\mathfrak{l}^\mathfrak{sp}(V)\otimes \mathbb K^N)\oplus(\mathfrak{n}(V)\otimes\wedge^2K^N).
 \end{equation}
Combining \eqref{tensorgl},\eqref{idspkNpos}, and the calculations of the subsections \ref{case1}-\ref{case5}, one can get more compact explicit formula for the algebraic prolongation of a given symplectic symbol.

In particular, from \eqref{splm} and \eqref{ll} it follows that
\begin{equation}
\label{idsplmN}
\mathfrak{l}^{\mathfrak{sp}}(\mathcal L_m^{\mathfrak{sp}}\otimes\mathbb K^N)\cong \wedge^2K^N\cong\mathfrak{so}(N).
\end{equation}

 Now assume that $K=\mathbb R$. Let $V$ be a nice symplectic $\sll_2$-module  such that as a graded space it coincides with $\mathcal L_m^{\mathfrak{sp}}\otimes \mathbb R^{N_++N_-}$ for some nonnegative integers $N_+$ and $N_-$ and the degree $-1$ component of the corresponding embedding of $\sll_2$ into $\mathfrak {sp}(V)$ is generated by the direct sum of $\tau_m^{\mathfrak{sp}}$ taken $N_+$ times and $-\tau_m^{\mathfrak{sp}}$ taken $N_-$ times. Then similarly to above it can be shown that

\begin{equation}
\label{idsplmN+-}
\mathfrak{l}(V)\cong\mathfrak{so}(N_+, N_-).
\end{equation}

Further, fix two functions $N_+, N_-:\frac{1}{2}\mathbb Z_{\rm {mod}}\mapsto \mathbb N\cup\{0\}$ with finite support and assume that the symplectic symbol $\mathfrak m$ is generated by the direct sum of endomorphisms of type $\tau_m^{\mathfrak{sp}}$ and $-\tau_m^{\mathfrak{sp}}$, where $\tau_m^{\mathfrak{sp}}$ appears $N_+(m)$ times and
$-\tau_{m}^{\mathfrak{sp}}$ appears $N_-(m)$ times in this sum for each $m\in \mathbb N$. These symbols correspond to curves in a Lagrangian Grassmannian satisfying condition (G) in the terminology of the previous papers of the second author with C. Li (\cite{zelli1,zelli2}).
Then from \eqref{idsplmN+-} and \eqref{ll} it follows that the non-effectiveness ideal of $\mathfrak u(\mathfrak m)$ is equal to $\displaystyle{\bigoplus_{m\in\mathbb N} \mathfrak{so}(N_+(m), N_-(m))}$ and it is concentrated in the zero component of $\mathfrak u(\mathfrak m)$. From this and Remark \ref{parprolongrem} it follows that in the parameterized case $\mathfrak u_+(\mathfrak m)=\displaystyle{\bigoplus_{m\in\mathbb N} \mathfrak{so}(N_+(m), N_-(m))}$ and, by Theorem \ref{mainthm} and Remark \ref{princrem}, if the normalization condition is fixed, then for any parametrized curve of flags with symbol $\mathfrak m$ there exists a unique principle bundle of moving frames with the structure group ${\displaystyle{\prod_{m\in\mathbb N}} O(N_+(m), N_-(m))}$ and the unique principle Ehresmann connection such that this connection satisfies the chosen normalization condition. This result was proved
in \cite{zelli1,zelli2} ( Theorem 1 in \cite{zelli1} and Theorems 1 and 3 in \cite{zelli2}) for specifically chosen normalization conditions. So, our main Theorem \ref{mainthm} gives more conceptual point of view on the constructions of those papers and clarifies them algebraically.

Finally, formulas \eqref{idspkNpos}-\eqref{idsplmN+-} hold true if $V$ is a graded orthogonal space and $\mathfrak{sp}$ is replaced by $\mathfrak{so}$.
Also, the same conclusion as in the previous paragraph can be done if an orthogonal symbol  is generated by the direct sum of endomorphisms of type $\tau_m^{\mathfrak{so}}$ and $-\tau_m^{\mathfrak{so}}$.

%


\begin{thebibliography}{99}



\bibitem{agrgam1}
A.A. Agrachev, R.V. Gamkrelidze, \emph{Feedback-invariant optimal
control theory - I. Regular extremals}, J. Dynamical and
Control Systems, {\bf 3}(1997), No. 3, 343-389.

\bibitem{agrachev}
A.A. Agrachev, Feedback-invariant optimal control theory - II.
Jacobi Curves for Singular Extremals, J. Dynamical and Control
Systems, {\bf 4}(1998), No. 4 , 583-604.

\bibitem{jac1}
A. Agrachev, I. Zelenko, {\sl Geometry of Jacobi curves.
I}, J. Dynamical and Control systems, {\bf 8}(2002),No. 1, 93-140.


\bibitem{jac2} A. Agrachev, I. Zelenko, {\sl Geometry of Jacobi curves
. II}, J. Dynamical and Control systems,{\bf 8}(2002), No. 2,
167-215.

\bibitem{princjac}
A. Agrachev, I. Zelenko, {\sl Principle Invariants of
Jacobi curves}, In the book: Nonlinear Control in the Year 2000,
v.1, A. Isidori, F. Lamnabhi-Lagarrigue $\&$ W.Respondek, Eds,
Lecture Notes in Control and Information Sciences 258, Springer,
2001, 9-22.

\bibitem{herm}
M. Ackerman and R. Hermann, \emph{Sophus Lie's 1880 Transformation Group Paper},
Math Sci Press 1975.

\bibitem{quiv} H. Derksen, J. Weyman, {\sl Quiver Representations},
Notices of AMS,  {\bf 52}, 2005, Number 2, 200-206.

\bibitem{cartan}
E. Cartan, \emph{La th\`eorie des groupes finis et continus et la g\`eom\`etrie diff\`erentielle trait\'es par la m\'ethode du rep\'ere mobile}, Cahiers Scientifiques, Vol. 18, Gauthier-Villars,Paris, 1937.

\bibitem{doubproj}
B. Doubrov, \emph{Projective parametrization of homogeneous curves}, Arch. Math. (Brno) 41 (2005), no. 1, 129--133.

\bibitem{doub3} B.~Doubrov, \emph{Generalized Wilczynski invariants for non-linear ordinary differential equations},
In: Symmetries and Overdetermined Systems of Partial Differetial Equations, IMA
Volume 144, Springer, NY, 2008, pp.~25--40

\bibitem{doubkom}
B. M. Doubrov, B. P. Komrakov, \emph{Classification of homogeneous
submanifolds in homogeneous spaces,} Lobachevskii J. Math., 1999,
Volume 3, Pages 19-38.

\bibitem{doubmor}
B. Doubrov, Y. Machida, and T. Morimoto, \emph{Linear equations on filtered manifolds
and submanifolds of flag varieties}, preprint.

\bibitem{doubzel1}
B. Doubrov, I. Zelenko, \emph{ A canonical frame for nonholonomic
rank two distributions of maximal class},  C.R. Acad. Sci. Paris,
Ser. I, Vol. 342, Issue 8 (15 April 2006), 589-594 ,

\bibitem{doubzel2}
B. Doubrov, I. Zelenko, \emph{ On local geometry of nonholonomic
rank 2 distributions}, Journal of London Mathematical Society,
80(3):545-566, 2009;

\bibitem{doubzel3}
B. Doubrov, I. Zelenko, \emph{ On local geometry of rank 3
distributions with 6-dimensional square}, submitted,
arXiv:0807.3267v1[math. DG], 40 pages.


\bibitem{dzpar} B. Doubrov, I. Zelenko, \emph{Geometry of curves in parabolic homogeneous spaces}, preprint, September 2011.

\bibitem{eastslov}
M. Eastwood, J. Slov\'{a}k, \emph{Preferred parameterisations on homogeneous curves}, Comment. Math. Univ. Carolin. 45 (2004), no. 4, 597--606.

\bibitem{olver1}
M.~Fels, P.J. ~Olver, \emph{Moving coframes. I. A practical algorithm}, Acta Appl. Math. 51 (1998) 161-213

\bibitem{olver2}
M.~Fels, P.J. ~Olver, \emph{Moving coframes II: Regularization and Theoretical Foundations}, Acta Appl. Math. 55 (1999) 127-208

\bibitem{Fulton} W. Fulton, J.~Harris, \emph{Representation theory: a first
course}, Springer--Verlag, NY, 1991.

\bibitem{gabriel}
P. Gabriel,
\emph{Unzerlegbare Darstellungen. I. (German. English summary)}
Manuscripta Math. 6 (1972), 71-103; correction, ibid. 6 (1972), 309.

\bibitem{gelf} I.M Gelfand, \emph{Lectrures on linear algebra}, New York, Interscience Publishers, 1961, 185 pages

\bibitem{green}
M.L. Green, \emph{The moving frame, differential invariants and rigidity theorem for curves in homogeneous spaces}, Duke Mathematics Journal 45 (1978),  735?779

\bibitem{griffiths}
P.A. Griffiths, \emph{On Cartan?s method of Lie groups and moving frames as applied to uniqueness and existence questions in differential geometry}, Duke Math. J. 41(1974), 775?814.

\bibitem{hump} J.E. Humphreys, \emph{Introduction to Lie Algebras and Representation Theory,} Third  printing, revised. Graduate Text in Mathematics, 9. Springer-Verlag, New York-Berlin, 1980, xii+171 pp

\bibitem{jacob} N. Jacobson, \emph{Lie algebras}, Intersci. Tracts in Pure and Appl. Math., 10, New=York--London: John Wiley and Sons, 1962,


\bibitem{slie} S. Lie, \emph{Theory der Transformationgruppen}, Bd. 3, Leipzig:Teubner, 1893

\bibitem{beffa}
G. Mari Beffa,
\emph{Poisson brackets associated to the conformal geometry of curves}, Trans. Amer. Math. Soc. 357 (2005), p. 2799-282

\bibitem{beffa1}
G. Mari Beffa,
\emph{On completely integrable geometric evolutions of curves of Lagrangian planes.}
[J] Proc. R. Soc. Edinb., Sect. A, Math. 137, No. 1, 111-131 (2007).

\bibitem{beffa2}
G. Mari Beffa,
\emph{Projective-type differential invariants and geometric curve evolutions of KdV-type in flat homogeneous manifolds}
Annales de l'institut Fourier, 58 no. 4 (2008), p. 1295-1335.

\bibitem{beffa3}
G. Mari Beffa, \emph{Moving frames, Geometric Poisson brackets and the KdV-Schwarzian evolution of pure spinors}, Annales de l'Institut Fourier, accepted for publication.

\bibitem{ovs} V. Ovsienko,
{\emph Lagrange schwarzian derivative and symplectic Sturm theory,}
Annales de la faculta des sciences de Toulouse S�. 6, 2 no. 1
(1993), p. 73-96


\bibitem{seashi1}
Y. Se-ashi, \emph{A geometric construction of Laguerre-Forsyth's
canonical forms of linear ordinary differential equations.}
Shiohama, K. (ed.), Progress in differential geometry. Tokyo:
Kinokuniya Company Ltd. Adv. Stud. Pure Math. 22, 265-297 (1993)

\bibitem{seashi2}
Y. Se-ashi, \emph{On differential invariants of integrable finite
type linear differential equations}, Hokkaido Math. J. 17 (1988),
151--195.

\bibitem{tan} N.~Tanaka, \emph{On differential systems, graded Lie
  algebras and pseudo-groups}, J.\ Math.\ Kyoto.\ Univ., \textbf{10}
  (1970), 1--82.
\bibitem{tan2} N. Tanaka, \emph{ On the equivalence problems associated with simple graded Lie
algebras}, Hokkaido Math. J.,{\bf 6} (1979), 23--84.

\bibitem{wil} E.J.~Wilczynski, \emph{Projective differential geometry of
curves and ruled surfaces}, Leipzig, Teubner, 1905.

\bibitem{vinb}
\'E. B. Vinberg, \emph{The Weyl group of a graded Lie algebra.} (Russian)  Izv. Akad. Nauk SSSR Ser. Mat.  40  (1976), no. 3, 488--526, 709.

\bibitem{vinberg2}
\'E. B. Vinberg,\emph{Classification of homogeneous nilpotent elements of a semisimple graded Lie algebra}, Sel. Math. Sov., v. 6, 1987, 15--35.

\bibitem{zelrank1} I. Zelenko, {\sl  Complete systems of invariants for rank 1 curves in
Lagrange Grassmannians}, Differential Geom. Application, Proc. Conf.
Prague, 2005, pp 365-379, Charles University, Prague

\bibitem{zelli1} I. Zelenko, C. Li, \emph{Parametrized curves in Lagrange Grassmannians}, C.R. Acad. Sci. Paris, Ser. I, Vol. 345, Issue 11, 647-652.

\bibitem{zelli2} I. Zelenko, C. Li, \emph{Differential geometry of curves in Lagrange Grassmannians with given Young diagram}, Differential Geometry and its Applications, 27 (2009), 723-742
\end{thebibliography}
\end{document}